\documentclass[11pt]{article}

\usepackage{amsmath}
\usepackage{amssymb}
\usepackage{amsfonts}
\usepackage{mathrsfs}
\usepackage{color}
\usepackage{float}
\usepackage{caption}
\usepackage[linktocpage,colorlinks,linkcolor=blue,anchorcolor=blue, citecolor=blue,urlcolor=blue]{hyperref}

\newtheorem{theorem}{Theorem}[section]
\newtheorem{definition}[theorem]{Definition}
\newtheorem{lemma}[theorem]{Lemma}
\newtheorem{corollary}[theorem]{Corollary}
\newtheorem{proposition}[theorem]{Proposition}
\newtheorem{algorithm}[theorem]{Algorithm}
\numberwithin{equation}{section}

\newcommand{\BB}{\mathbb{B}}

\newcommand{\BD}{\mathbb{D}}
\newcommand{\BE}{\mathbb{E}}
\newcommand{\BH}{\mathbb{H}}
\newcommand{\BX}{\mathbb{X}}
\newcommand{\BY}{\mathbb{Y}}
\newcommand{\BZ}{\mathbb{Z}}
\newcommand{\BS}{\mathbb{S}}
\newcommand{\BT}{\mathbb{T}}
\newcommand{\BN}{\mathbb{N}}
\newcommand{\BI}{\mathbb{I}}
\newcommand{\I}{\mathbb{I}}
\newcommand{\R}{\mathbb{R}}

\newcommand{\X}{\mathcal{X}}
\newcommand{\Y}{\mathcal{Y}}
\newcommand{\Z}{\mathcal{Z}}

\newcommand{\TT}{\mathcal{T}}
\newcommand{\U}{\mathcal{U}}
\newcommand{\V}{\mathcal{V}}

\newcommand{\ba}{\boldsymbol{a}}

\newcommand{\e}{\boldsymbol{e}}
\newcommand{\bx}{\boldsymbol{x}}
\newcommand{\by}{\boldsymbol{y}}
\newcommand{\bz}{\boldsymbol{z}}
\newcommand{\bu}{\boldsymbol{u}}
\newcommand{\bv}{\boldsymbol{v}}

\newcommand{\T}{\textnormal{T}}
\newcommand{\F}{\textnormal{F}}

\newcommand{\Prob}{\textnormal{Prob}\,}

\newcommand{\sign}{\textnormal{sign}\,}
\newcommand{\ex}{{\bf\sf E}}

\begin{document}

\title{Approximating Tensor Norms via Sphere Covering: Bridging the Gap Between Primal and Dual}

\author{
Simai HE
\thanks{Research Institute for Interdisciplinary Sciences, School of Information Management and Engineering, Shanghai  University of Finance and Economics, Shanghai 200433, China. Email: simaihe@mail.shufe.edu.cn}
    \and
Haodong HU
\thanks{Department of Computer Science and Technology, School of Information Management and Engineering, Shanghai University of Finance and Economics, Shanghai 200433, China. Email: hu.haodong@shufe.edu.cn}
    \and
Bo JIANG
\thanks{Research Institute for Interdisciplinary Sciences, School of Information Management and Engineering, Shanghai  University of Finance and Economics, Shanghai 200433, China. Email: isyebojiang@gmail.com}
    \and
Zhening LI
\thanks{School of Mathematics and Physics, University of Portsmouth, Portsmouth PO1 3HF, United Kingdom. Email: zheningli@gmail.com}
}

\date{\today}

\maketitle

\begin{abstract}

The matrix spectral and nuclear norms appear in enormous applications. The generalizations of these norms to higher-order tensors is becoming increasingly important but unfortunately they are NP-hard to compute or even approximate. Although the two norms are dual to each other, the best known approximation bound achieved by polynomial-time algorithms for the tensor nuclear norm is worse than that for the tensor spectral norm. In this paper, we bridge this gap by proposing deterministic algorithms with the best bound for both tensor norms.
Our methods not only improve the approximation bound for the nuclear norm, but are also data independent and easily implementable comparing to existing approximation methods for the tensor spectral norm.
The main idea is to construct a selection of unit vectors that can approximately represent the unit sphere, in other words, a collection of spherical caps to cover the sphere. For this purpose, we explicitly construct several collections of spherical caps for sphere covering with adjustable parameters for different levels of approximations and cardinalities. These readily available constructions are of independent interest as they provide a powerful tool for various decision making problems on spheres and related problems. We believe the ideas of constructions and the applications to approximate tensor norms can be useful to tackle optimization problems over other sets such as the binary hypercube.

\vspace{0.25cm}


\vspace{0.25cm}

\end{abstract}

\section{Introduction}\label{sec:introduction}

With the advances in data collection and storage capabilities, massive multidimensional and multiway tensor data are being generated in a wide range of emerging applications~\cite{KB09}. Tensor computations and optimizations have been an active research area in the recent decade. Computing tensor norms are evidently essential in modelling various tensor optimization problems. One typical example is tensor completion (see e.g.,~\cite{YZ16}) in which the tensor nuclear norm is commonly used as the {\color{black}convex surrogate} of the tensor rank. However, most tensor norms are NP-hard to compute~\cite{HL13}, such as the spectral norm~\cite{HLZ10} and the nuclear norm~\cite{FL18} when the order of a tensor is more than two, a sharp contrast to matrices (tensors of order two) whose spectral and nuclear norms are easy to compute, e.g., using singular value decompositions.

The tensor spectral norm~\cite{L05} is commonly known as the maximization of a multilinear form over Cartesian products of unit spheres, a standard higher-order generalization of the matrix spectral norm. Taking a tensor $\TT=(t_{ijk})\in\R^{n\times n\times n}$ of order three as an example, its spectral norm
\begin{equation}\label{eq:3tnorm}
\|\TT\|_{\sigma}= \max \left\{\TT(\bx,\by,\bz): \|\bx\|_2=\|\by\|_2=\|\bz\|_2=1,\,\bx,\by,\bz\in\R^n\right\},
\end{equation}
where
$
\TT(\bx,\by,\bz) = \sum_{i=1}^{n}\sum_{j=1}^{n}\sum_{k=1}^{n} t_{ijk} x_iy_jz_k
$
is a trilinear form of $(\bx,\by,\bz)$. This is equivalent to the best rank-one approximation of the tensor $\TT$ in the tensor community
\[\min\left\{\| \TT-\lambda\,\bx\otimes \by \otimes \bz \|_{\F}: \lambda\in\R, \|\bx\|_2=\|\by\|_2=\|\bz\|_2=1,\,\bx,\by,\bz\in\R^n\right\},
\]
where $\|\bullet\|_{\F}$ stands for the Frobenius norm and $\otimes$ stands for the vector outer product, meaning that $\bx\otimes \by \otimes \bz$ is a rank-one tensor.

Although the tensor spectral norm is NP-hard to compute, it is easy to obtain feasible solutions of~\eqref{eq:3tnorm} to approximate this norm. There have been a lot of research works~\cite{S11,ZQY12,LHZ12,HJLZ14,HS14} on approximation algorithms of~\eqref{eq:3tnorm} in the optimization community since the seminal work of He et al.~\cite{HLZ10}. The best known worst-case bound to approximate~\eqref{eq:3tnorm} in polynomial time is $\Omega\left({\sqrt\frac{\ln n}{n}}\right)$~\cite{S11,HJLZ14}. One simple approach for this bound is a naive randomized algorithm in~\cite{HJLZ14}:
\begin{enumerate}
  \item Sample a vector $\bv$ uniformly on the sphere\footnote{\color{black}The $n$ in $\BS^n$ refers to the dimension of the space in which this sphere of dimension $n-1$ lives.} $\BS^n:=\{\bx\in\R^n:\|\bx\|_2=1\}$ and compute the spectral norm of the resulting matrix, i.e., $\max_{\|\bx\|_2=\|\by\|_2=1}\TT(\bx,\by,\bv)$;
  \item Repeat the above procedure independently until the largest objective value from all samples hits the desired bound.
\end{enumerate}
If we were able to sample all vectors in the unit sphere for $\bz$, then this approach certainly finds $\max_{\|\bx\|_2=\|\by\|_2=\|\bz\|_2=1}\TT(\bx,\by,\bz)$. It is obviously not possible to cover the unit sphere by enumerating unit vectors. However, if we are allowed some tolerance, say an approximation ratio $\tau\in(0,1]$, then a sample unit vector $\bv$ becomes a spherical cap
\[\BB^n(\bv,\tau):=\left\{\bx\in\BS^n:\bx^{\T}\bv\ge\tau\right\}\]
with the angular radius $\theta=\arccos\tau$. In this setting, $\bv$ is able to generate a $\tau$-approximate solution if and only if the spherical cap $\BB^n(\bv,\tau)$ includes an optimal $\bz$ in~\eqref{eq:3tnorm}. Alternatively, if we have a collection of sample unit vectors whose corresponding spherical caps joining together covers the whole sphere, then the best one in this collection can generate a $\tau$-approximate solution. In fact, the above algorithm does imply a randomized cover of the sphere whose covering volume is at least $1-\epsilon$ for any $\epsilon>0$ with high probability. However, this is much weaker than what we need here and {\color{black} even cannot guarantee the existence of a full cover}. One of the major contributions in this paper is to find a reasonable number of spherical caps to cover the sphere, deterministically and explicitly. 

{\color{black}There are certainly lots of decision problems over spheres.} Among them many are hard problems that approximate solutions are commonly acceptable such as wireless communications~\cite{VH20} and spherical facility location~\cite{X95}. There are even harder problems {\color{black}where} sphere covering seems irrelevant but it can be indeed helpful. One of these problems is {\color{black}computing the tensor nuclear norm.}
Taking $\TT\in\R^{n\times n\times n}$ again as an example, its nuclear norm is 
\begin{equation}\label{eq:3nnorm}
\|\TT\|_*=\min\left\{\sum_{i=1}^r|\lambda_i| : \TT=\sum_{i=1}^r\lambda_i\, \bx_i\otimes\by_i\otimes\bz_i,\,\lambda_i\in\R,\, \|\bx_i\|_2=\|\by_i\|_2=\|\bz_i\|_2=1
,\, r\in\BN \right\}.
\end{equation}
{\color{black}The decomposition of $\TT$ into rank-one tensors in~\eqref{eq:3nnorm} is known as a CANDECOMP/PARAFAC (CP) decomposition~\cite{H27}}. {\color{black}While CP decompositions usually require the number of rank-one terms to be minimum, there is no such constraint in~\eqref{eq:3nnorm}.} In fact, the tensor nuclear norm and spectral norm are dual to each other (see e.g.,~\cite{LC14}), i.e.,
\[
\|\TT\|_*=\max_{\|\X\|_\sigma \le 1}\langle\TT,\X\rangle \mbox{ and } \|\TT\|_\sigma=\max_{\|\X\|_* \le 1}\langle\TT,\X\rangle,
\]
where $\langle,\rangle$ stands for the Frobenius inner product. Computing or approximating tensor nuclear norm is much harder no matter {\color{black}using the definition~\eqref{eq:3nnorm} or the dual formulation---the corresponding feasibility problem is not easy at all. The situation is different for the tensor spectral norm as the feasibility to~\eqref{eq:3tnorm} is trivial. There are various methods~\cite{DDV00,RK00,KB09,CHLZ12,VVM12,NW14,JMZ15,DCD16} to compute the tensor spectral norm in practice but there is only one known method~\cite{N17} to compute the tensor nuclear norm, to the best of our knowledge.} This crucial fact has resulted alternative concepts for the tensor nuclear norm in practice, such as the average nuclear norms of the matrix flattenings from three different ways. In terms of approximating the tensor nuclear norm, the best polynomial-time worst-case approximation bound is $\Omega\left(\frac{1}{\sqrt{n}}\right)$ via matrix flattenings~\cite{H15} or partitions into matrix slices~\cite{L16}. This bound is worse than the best known one $\Omega\left({\sqrt\frac{\ln n}{n}}\right)$ for the tensor spectral norm. It is natural to expect achieving this bound for the dual norm to the tensor spectral norm. As another major work in this paper, via certain reformulation and convex optimization proposed in~\cite{HJL22}, we are able to bridge the gap between the primal and dual norms, with the help of constructions of spherical caps for sphere covering.

Covering a sphere by identical spherical caps has been studied in computational geometry since the pioneering work of Rogers~\cite{R58}. Instead of describing spherical caps via the angular radius, the caps are measured in normalized volume in the study. Specifically, by defining the normalized volume of a spherical cap to be its true volume over the volume of $\BS^n$ (in this sense the normalized volume of $\BS^n$ is one), sphere covering asks for a given positive integer $m$, what is the smallest $\delta$ such that there are $m$ spherical caps with normalized volume $\delta$ covering $\BS^n$? The quantity $\delta m$ is called the density of the covering. Studying the bounds of this density has been the main research topic along this line. An upper bound of $O\left(n\ln n\right)$ for the covering density was obtained by Rogers~\cite{R63} for sufficiently small $\delta$. This remains the best known asymptotic upper bound although there were improvements made in terms of the constant of the asymptotic bound and for any $\delta$ in~\cite{BW03,D07}.

For the lower bound of covering density, there is not a clear answer in general other than the trivial one, i.e., $\delta m\ge 1$. Rogers~\cite{R58} stated that the density of a covering cannot beat a natural strategy based on tiling $\R^n$ with regular simplices, known as the simplex bound, whose value remains a conjecture and unproven. Rogers~\cite{R58} computed that for $\delta\rightarrow 0$ the density is close to $\frac{n}{e\sqrt{e}}$. It is believed that the density is $\Omega(n)$. Several special cases for the simplex bound have been confirmed, either for very small $\delta$ or for $\delta$ in large cap regime; see~\cite{JM20} and references therein. Other than the two trivial cases for $m=1$ and $m=2$ which correspond to $\delta=1$ and $\delta=\frac{1}{2}$, respectively, perhaps the first nontrivial work along this line is due to Lusternik and Schnirelmann~\cite{B06}: If $n$ open or closed sets cover $\BS^n$, then one of these contains a pair of antipodal points. This implies that if $\delta<\frac{1}{2}$ then $m\ge n+1$. An obvious lower bound of $\Omega(n)$ for any universal constant $\delta<\frac{1}{2}$.

There are two optimization problems that are relevant to sphere covering in the literature. The sphere coverage verification is to decide whether a given set of spherical caps cover the sphere or not. Petkovi\'{c} et a.~\cite{PPL12} showed that sphere coverage verification is NP-hard and proposed a recursive algorithm based on quadratic optimization. The spherical discrepancy is to find the furthest point in $\BS^n$ to a given set of points $\{\bv_1,\bv_2,\dots,\bv_m\}\subseteq \BS^n$, i.e., $\min_{\bx\in\BS^n}\max_{1\le i\le m}\bx^{\T}\bv_i$. The spherical discrepancy is also NP-hard since it is the optimization version of sphere coverage verification who is a decision version of spherical discrepancy. Jones and McPartlon~\cite{JM20} proposed a multiplicative weights-based algorithm that obtains an approximation bound up to lower order terms.

Although there is extensive research on the density of sphere covering and related problems, they do not exactly serve the purpose of our study in this paper. The asymptotic bounds on the normalized volumes are not aligned with the goal to obtain approximation bounds based on inner products between unit vectors. The upper bounds obtained in~\cite{BW03} are existence results via a randomized approach. The construction in~\cite{RS09} works only in the large cap regime for $\delta=e^{-\sqrt{n}}$ which resulted the number of caps to be exponential in $n$. A recent work on spherical discrepancy minimization~\cite{JM20} showed an algorithm to generate spherical caps sequentially until a covering is satisfied but the running time to generate a cap is $O(n^{10})$. Our goal is to achieve a good balance between the approximation measured by $\cos\theta$ for the angular radius $\theta$ and the number of caps that are not too large, say bounded by a polynomial function of $n$. More importantly, we hope to obtain explicit constructions of spherical caps to cover the unit sphere. These will be of great beneficial to the algorithm and optimization community apart from our applications in approximating tensor norms. The products of our simple and explicit constructions, together with some trivial and known constructions, are summarized in Table~\ref{table:list}.
\begin{table}[h]
\centering
\small\begin{tabular}{|p{2.6in}|p{1.3in}|p{1.3in}|}
\hline
Set of $\bv$'s for $\BB^n(\bv,\tau)$ & $\tau$ for $\BB^n(\bv,\tau)$  & Number of $\BB^n(\bv,\tau)$'s   \\
\hline
Any $\{\bv\}$ where $\bv\in\BS^n$         & $-1$ & $1$ \\
Any $\{\bv,-\bv\}$ where $\bv\in\BS^n$    & $0$ & $2$ \\
Any regular simplex inscribed in $\BS^n$  & $1/n$ & $n+1$ \\
Any basis of $\R^n$ with their negations  & $1/\sqrt{n}$ & $2n$ \\
    $\BH^n_1$ (Section~\ref{sec:h1}), $\BH^n_4$ and $\BH^n_5$ (Section~\ref{sec:h4}) & $\Omega\big(\sqrt{\ln n/ n}\big)$ & $O(n^\alpha)$ for $\alpha > 1$
\\
$\BH^n_2$ (Section~\ref{sec:h2})   & $\Omega\big(1/\sqrt{\ln n}\big)$  & $O(3^n)$ \\
$\BH^n_3$ (Section~\ref{sec:h3})   & $\Omega\big(1\big)$ & $O(\beta^n)$ for $\beta > 4$\\
Grid points in spherical coordinates & $1-O\big(n/m^2\big)$ & $O(m^{n-1})$ \\
\hline
\end{tabular}
\caption{Constructions of spherical caps to cover the unit sphere}\label{table:list}
\end{table}

This paper is organized as follows. After introducing some uniform notations, we propose various constructions of spherical caps for sphere covering and bound the ratio $\tau$ and number of caps of each construction in Section~\ref{sec:h}. We work around a key ratio $\Omega\left(\sqrt{\frac{\ln n}{n}}\right)$ which is the largest possible if the number of spherical caps is $O(n^\alpha)$ for some universal constant $\alpha>1$, from randomization (Section~\ref{sec:h1}) to deterministic covering (Section~\ref{sec:h4}) with some interesting byproducts (Sections~\ref{sec:h2} and~\ref{sec:h3}). In Section~\ref{sec:tensor}, we apply the covering results to approximate tensor norms. Specifically, we propose the first implementable and deterministic algorithm with the known best approximation bound for the tensor spectral norm and related polynomial optimization problems in Section~\ref{sec:snorm}. A deterministic algorithm with an improved approximation bound for the tensor nuclear norm is proposed in Section~\ref{sec:nnorm}. Numerical performance of the proposed algorithms are reported in Section~\ref{sec:numerical}. Finally, some concluding remarks are given in Section~\ref{sec:remark}.

\subsection*{Some uniform notations}

Throughout this paper we uniformly adopt lowercase letters (e.g., $x$), boldface lowercase letters (e.g., $\bx=\left(x_i\right)$), capital letters (e.g., $X=\left(x_{ij}\right)$), and calligraphic letters (e.g., $\X=\left(x_{i_1i_2\dots i_d}\right)$) to denote scalars, vectors, matrices, and higher-order (order three or more) tensors, respectively. Denote $\R^{n_1\times n_2\times\dots\times n_d}$ to be the space of real tensors of order $d$ with dimension $n_1\times n_2\times\dots\times n_d$. The same notation applies for a vector space and a matrix space when $d=1$ and $d=2$, respectively. Denote $\BN$ to be the set of positive integers.

The Frobenius inner product between two tensors $\U,\V\in\R^{n_1\times n_2\times\dots\times n_d}$ is defined as
\[
\langle\U, \V\rangle := \sum_{i_1=1}^{n_1}\sum_{i_2=1}^{n_2} \dots\sum_{i_d=1}^{n_d} u_{i_1i_2\dots i_d} v_{i_1i_2\dots i_d}.
\]
Its induced Frobenius norm is naturally defined as $\|\TT\|:=\sqrt{\langle\TT,\TT\rangle}$. The two terms automatically apply to tensors of order two (matrices) and tensors of order one (vectors) as well. This is the conventional norm (a norm without a subscript) used throughout the paper.

All blackboard bold capital letters denote sets, such as $\R^n$, the unit sphere $\BS^n$, a spherical cap $\BB^n(\bv,\tau)$, the standard basis $\BE^n:=\{\e_1,\e_2,\dots,\e_n\}$ of $\R^n$, where the superscript $n$ indicates that the concerned set is a subset of $\R^n$. Three vector operations are used, namely the outer product $\otimes$, the Kronecker product $\boxtimes$, and appending vectors $\vee$. Specifically, if $\bx\in\R^{n_1}$ and $\by\in\R^{n_2}$, then
\begin{align*}
\bx\otimes\by&=\bx\by^{\T}\in\R^{n_1\times n_2} \\
\bx\boxtimes\by&=(x_1\by^{\T},x_2\by^{\T},\dots,x_{n_1}\by^{\T})^{\T}\in\R^{n_1n_2} \\
\bx\vee\by&=(x_1,x_2,\dots,x_{n_1},y_1,y_2,\dots,y_{n_2})^{\T}\in\R^{n_1+n_2}.
\end{align*}
These three operators also apply to vector sets via element-wise operations.

As a convention, the notion $\Omega(f(n))$ means that there are positive universal constants $\alpha,\beta$ and $n_0$ such that $\alpha f(n)\le\Omega(f(n))\le\beta f(n)$ for all $n\ge n_0$, i.e., the same order of magnitude to $f(n)$.



\section{Sphere covering by spherical caps}\label{sec:h}

This section is devoted to explicit constructions of spherical caps to cover $\BS^n$ in $\R^n$ for $n\ge2$. 
Although this is more commonly denoted by $\BS^{n-1}$ in the literature, our notation is to emphasize that the sphere resides in the space of $\R^n$ and to better understand the constructions via Kronecker products.

Recall that for $\bv\in\BS^n$ and $-1\le\tau\le 1$, $\BB^n(\bv, \tau)=\left\{\bx\in\BS^n: \bx^{\T}\bv \ge \tau\right\}$ is a closed spherical cap with the angular radius $\arccos\tau$. Obviously, $\BB^n(\bv, -1)=\BS^n$, $\BB^n(\bv, 0)$ is a hemisphere, and $\BB^n(\bv, 1)$ is a single point. A set of unit vectors $\BH^n=\{\bv_i\in\BS^n:i=1,2,\dots,m\}$ is called a $\tau$-hitting set with cardinality $m$ if $\bigcup_{i=1}^m \BB^n\left(\bv_i, \tau\right) = \BS^n$, i.e., the $m$ spherical caps cover the unit sphere. Denote all $\tau$-hitting sets of $\BS^n$ with cardinality no more than $m$ to be
\[
\BT(n,\tau,m):=\left\{\BH^n\subseteq\BS^n: \BH^n \mbox{ is a $\tau$-hitting set},\,|\BH^n|\le m\right\}.
\]
It is easy to see the monotonicity, i.e.,
\[
\begin{array}{ll}
\BT(n,\tau_2,m)\subseteq\BT(n,\tau_1,m) &\mbox{if } \tau_1\le\tau_2 \\
\BT(n,\tau,m_1)\subseteq\BT(n,\tau,m_2) &\mbox{if } m_1\le m_2.
\end{array}
\]

We will be working around $\tau$-hitting sets with $\tau=\Omega\left(\sqrt{\frac{\ln n}{n}}\right)$. This is the largest possible if the cardinality of the hitting set is bounded by $O(n^\alpha)$ with some universal constant {\color{black}$\alpha>1$}; see e.g.~\cite{HJLZ14}. Other useful $\tau$-hitting sets with larger $\tau$'s are also constructed as byproducts that are of independent interest. The aim is to construct hitting sets with the cardinality as small as possible. Let us first look at some elementary ones.

It is obvious that for any $\bv\in\BS^n$, 
\[
\{\bv\}\in\BT(n,-1,1) \mbox{ and } \{\bv,-\bv\}\in\BT(n,0,2)
\]
both attaining the minimum cardinality. For $\tau>0$, the famous Lusternik-Schnirelmann theorem~\cite{B06} rules out any possible $\tau$-hitting set with cardinality no more than $n$. There is an elegant construction of $\frac{1}{n}$-hitting sets with cardinality $n+1$. If $\bv_1,\bv_2,\dots,\bv_{n+1}$ are the vertices of a regular simplex centered at the origin and inscribed in $\BS^n$, then
\[
\left\{\bv_1,\bv_2,\dots,\bv_{n+1}\right\}\in\BT\left(n,\frac{1}{n},n+1\right).
\]
Detailed construction is easier to be obtained from $\R^{n+1}$ and is left to interested readers.
Raising $\tau$ to $\frac{1}{\sqrt{n}}$ without increasing the number of vectors too much, one has for any basis $\{\bv_1,\bv_2,\dots,\bv_n\}$ of $\R^n$,
\[
\left\{\pm\bv_1,\pm\bv_2,\dots,\pm\bv_n\right\}\in\BT\left(n,\frac{1}{\sqrt{n}},2n\right).
\]
However, slightly increasing this threshold, say to $\sqrt{\frac{\ln n}{n}}$, will significantly increase the cardinality of a hitting set. As mentioned earlier, if the cardinality is bounded by a polynomial function of $n$, then the largest possible $\tau=\Omega\left(\sqrt{\frac{\ln n}{n}}\right)$.

Toward the extreme case that $\tau$ is close to one, the longitude and latitude of the Earth provide a clue. For any $\bx=(x_1,x_2,\dots,x_n)^{\T}\in\BS^n$, we denote its spherical coordinates to be $(\varphi_1,\varphi_2,\dots,\varphi_{n-1})$ with $\varphi_1,\varphi_2,\dots,\varphi_{n-2}\in[0,\pi]$ and $\varphi_{n-1}\in[0,2\pi)$ such that
\[ \begin{aligned}x_{1}&=\cos\varphi _{1}\\x_{2}&=\sin\varphi _{1}\cos\varphi _{2}\\ x_{3}&=\sin\varphi _{1}\sin\varphi _{2}\cos\varphi _{3}\\
&\;\;\vdots \\x_{n-1}&=\sin\varphi _{1}\dots \sin\varphi _{n-2}\cos\varphi _{n-1}\\x_{n}&=\sin\varphi _{1}\dots \sin\varphi _{n-2}\sin\varphi _{n-1}.\end{aligned}\]
If we let $\BD_1=\left\{\frac{k\pi}{m}:k=0,1,\dots,m-1\right\}$ and $\BD_{2}=\left\{\frac{k\pi}{m}:k=0,1,\dots,2m-1\right\}$, then the grid points in spherical coordinates (see~\cite[Lemma 3.1]{HJL22}) {\color{black}are}
\begin{equation}\label{eq:grid}
\BH^n_0(m):=\left\{\bx\in\BS^n: \varphi_1,\varphi_2,\dots,\varphi_{n-2}\in\BD_1,\,\varphi_{n-1}\in\BD_2 \right\} \in \BT\left( n,1-\frac{\pi^2(n-1)}{8m^2},2m^{n-1} \right).
\end{equation}
{\color{black}To see why $\BH^n_0(m)$ is such a hitting set.} For any $\bz\in\BS^n$ with spherical coordinates $\varphi(\bz)$, there must exist $\bx\in\BH^n_0(m)$ with spherical coordinates $\varphi(\bx)$, such that
$$
\|\bx-\bz\| \le \|\varphi(\bx)-\varphi(\bz)\|\le \frac{1}{2}\cdot\frac{\pi}{m}\cdot\sqrt{n-1}=\frac{\pi\sqrt{n-1}}{2m}.
$$
Since $\|\bx\|=\|\bz\|=1$, the above further leads to
$$
\bx^{\T}\bz = \frac{1}{2}\left(2-\|\bx-\bz\|^2\right)\ge \frac{1}{2}\left(2-\frac{\pi^2(n-1)}{4m^2}\right)=
1-\frac{\pi^2(n-1)}{8m^2}.
$$

\subsection{Randomized $\Omega\big(\sqrt{\ln n/n}\big)$-hitting sets}\label{sec:h1}

It is instructive to consider randomized hitting sets via the uniform distribution on $\BS^n$. This is also important as it guarantees the existence of $\Omega\left(\sqrt{\frac{\ln n}{n}}\right)$-hitting sets. The following probability bound (see e.g.,~\cite{HJLZ14,BGKKLS98}) provides an insight of such a hitting set.
\begin{lemma} \label{thm:inner}
For any $\gamma\in(0,\frac{n}{\ln n})$, if $\bu$ and $\bv$ are drawn independently and uniformly on $\BS^n$, then there is a constant $\delta_\gamma$ depending on $\gamma$ only, such that
\[
\Prob\left\{\bu^{\T}\bv\ge\sqrt{\frac{\gamma\ln n}{n}}\right\}\ge\frac{\delta_\gamma}{n^{2\gamma}\sqrt{\ln n}}.
\]
\end{lemma}
In fact, it is not difficult to cover $1-\epsilon$ of the volume of the unit sphere for any $\epsilon>0$ by applying Lemma~\ref{thm:inner} with the union bound; see~\cite{HJLZ14}. However, this is a much weaker statement than Theorem~\ref{thm:h1} below. {\color{black}In particular, the event of covering $1-\epsilon$ of the volume of $\BS^n$ for any given $\epsilon >0$ does not even guarantee the existence of a full cover, hence being weaker than the latter event.} The following randomized hitting set has a cardinality $O(n^\alpha)$ for some constant $\alpha>1$.

\begin{theorem}\label{thm:h1}
For any $\epsilon>0$ and $\gamma\in(0,\frac{n}{\ln n})$, there is a constant $\kappa_\gamma>0$ depending on $\gamma$ only, such that
\[\BH^n_1(\gamma,\epsilon):=\left\{\bz_i \mbox{ is i.i.d. uniform on $\BS^n$ for }i=1,2,\dots, \left\lceil \kappa_\gamma n^{2\gamma}\sqrt{\ln n}\left(n\ln n+\ln\frac{1}{\epsilon}\right)\right\rceil \right\}\]
satisfies
\[\Prob\left\{\BH^n_1(\gamma,\epsilon)\in\BT\left(n, \sqrt{\frac{\gamma\ln n}{2n}}, \left\lceil \kappa_\gamma n^{2\gamma}\sqrt{\ln n}\left(n\ln n+\ln\frac{1}{\epsilon}\right)\right\rceil \right)\right\}\ge1-\epsilon.\]
\end{theorem}
\begin{proof}
The sphere covering is established in two steps, a spherical grid $\BH^n_0$ to cover the whole sphere and the randomized hitting set $\BH^n_1$ to cover the grid.

According to~\eqref{eq:grid} one has $\BH^n_0(m)\in\BT\left( n,1-\frac{\pi^2(n-1)}{8m^2},2m^{n-1} \right)$. Let $m\ge n$. For any $\bx\in\BS^n$, there exists $\by\in \BH^n_0(m)$ such that $\bx^{\T}\by\ge 1-\frac{\pi^2(n-1)}{8m^2}$. By Lemma~\ref{thm:inner}, for any $\bz_i\in\BH^n_1(\gamma,\epsilon)$, there exists an $\delta_\gamma$ depending on $\gamma$ and
$\Prob\left\{\by^{\T}\bz_i \ge \sqrt{\frac{\gamma\ln n}{n}}\right\} \ge \frac{\delta_\gamma}{n^{2\gamma}\sqrt{\ln n}}$,
i.e., $\Prob\left\{\by^{\T}\bz_i < \sqrt{\frac{\gamma\ln n}{n}}\right\} \le 1-\frac{\delta_\gamma}{n^{2\gamma}\sqrt{\ln n}}$. 
Denote $t=|\BH^n_1(\gamma,\epsilon)|$. By the independence of $\bz_i$'s, we have
\[
\Prob\left\{\by\notin\bigcup_{i=1}^t\BB^n\left(\bz_i,\sqrt{\frac{\gamma\ln n}{n}}\right)\right\} = \Prob\left\{\max_{1\le i\le t}\by^{\T}\bz_i < \sqrt{\frac{\gamma\ln n}{n}}\right\} \le \left(1-\frac{\delta_\gamma}{n^{2\gamma}\sqrt{\ln n}}\right)^t.
\]
Since $\left|\BH^n_0(m)\right|=2m^{n-1}$ and the points of $\BH^n_0(m)$ are fixed, the probability that $\bigcup_{i=1}^t\BB^n\left(\bz_i,\sqrt{\frac{\gamma\ln n}{n}}\right)$ fails to cover at least one point of $\BH^n_0(m)$ is no more than $2m^{n-1}\left(1-\frac{\delta_\gamma}{n^{2\gamma}\sqrt{\ln n}}\right)^t$. In other words,
\begin{align*}
  \Prob\left\{\BH^n_0(m)\subseteq\bigcup_{i=1}^t\BB^n\left(\bz_i,\sqrt{\frac{\gamma\ln n}{n}}\right)\right\}
  \ge 1-2m^{n-1}\left(1-\frac{\delta_\gamma}{n^{2\gamma}\sqrt{\ln n}}\right)^t.
\end{align*}
By noticing that $m\ge n\ge2$, it is not difficulty to verify that if $t\ge \frac{n^{2\gamma}\sqrt{\ln n}}{\delta_\gamma}\left(n\ln m+\ln\frac{1}{\epsilon}\right)$, then the right hand side of the above is at least $1-\epsilon$.

To summarize, if $\BH^n_0(m)\subseteq\bigcup_{i=1}^t\BB^n\left(\bz_i,\sqrt{\frac{\gamma\ln n}{n}}\right)$, then for any $\bx\in\BS^n$, there exists $\by\in\BH^n_0(m)$ such that $\by^{\T}\bx\ge 1-\frac{\pi^2(n-1)}{8m^2}$ and further there exists $\bz\in\BH^n_1(\gamma,\epsilon)$ such that $\bz^{\T}\by\ge\sqrt{\frac{\gamma\ln n}{n}}$. If we are able to verify $\bz^{\T}\bx\ge\sqrt{\frac{\gamma\ln n}{2n}}$, then we must have $\bigcup_{i=1}^t\BB^n\left(\bz_i,\sqrt{\frac{\gamma\ln n}{2n}}\right)=\BS^n$. This finally leads to
\[
\Prob\left\{ \bigcup_{i=1}^t\BB^n\left(\bz_i,\sqrt{\frac{\gamma\ln n}{2n}}\right)=\BS^n \right\}
\ge  \Prob\left\{\BH^n_0(m)\subseteq\bigcup_{i=1}^t\BB^n\left(\bz_i,\sqrt{\frac{\gamma\ln n}{n}}\right)\right\} \ge 1-\epsilon.
\]

In order to show that $\bz^{\T}\bx\ge\sqrt{\frac{\gamma\ln n}{2n}}$, we let $\theta_1=\arccos (\by^{\T}\bx)$ and $\theta_2=\arccos (\bz^{\T}\by)$. Since $\cos\theta_1\ge1-\frac{\pi^2(n-1)}{8m^2}\ge 1-\frac{3}{2m}$, one has
$|\sin\theta_1| \le \sqrt{1-\left(1-\frac{3}{2m}\right)^2}\le\sqrt{\frac{3}{m}}$. Therefore,
\begin{equation}\label{eq:randombound}
\bz^{\T}\bx \ge \cos(\theta_1+\theta_2)
=\cos\theta_1\cos\theta_2-\sin\theta_1\sin\theta_2
\ge \left(1-\frac{3}{2m}\right)\cdot \sqrt{\frac{\gamma\ln n}{n}}-\sqrt{\frac{3}{m}}\cdot 1\ge \sqrt{\frac{\gamma\ln n}{2n}}
\end{equation}
if $n\ge n_0$ for some $n_0$ that depends on $\gamma$ only. By choosing $m=n$ in $\BH^n_0(m)$ and $\kappa_\gamma=\frac{1}{\delta_\gamma}$ and we have the desired $t$ for $n\ge n_0$.

To finish the final piece for remaining $n\le n_0$, we may enlarge $m$ in $\BH^n_0(m)$ in order for~\eqref{eq:randombound} to hold. If we choose $\kappa_\gamma= \frac{\ln m}{\delta_\gamma \ln n}$ correspondingly, this will ensure
\[
\kappa_\gamma n^{2\gamma}\sqrt{\ln n}\left(n\ln n+\ln\frac{1}{\epsilon}\right)= \frac{\ln m}{\ln n}\cdot\frac{n^{2\gamma}\sqrt{\ln n}}{\delta_\gamma }\left(n\ln n+\ln\frac{1}{\epsilon}\right) \ge \frac{n^{2\gamma}\sqrt{\ln n}}{\delta_\gamma}\left(n\ln m+\ln\frac{1}{\epsilon}\right).
\]
The largest $\kappa_\gamma$ for these finite $n\le n_0$ provides the final $\kappa_\gamma$ that depends only on $n_0$ who itself depends on $\gamma$ only.
\end{proof}

Theorem~\ref{thm:h1} not only provides a simple construction with varying $\gamma$ but also trivially implies the existence of hitting sets in $\BT\left(n, \Omega\left(\sqrt{\frac{\ln n}{n}}\right), O(n^\alpha)\right)$. Although $\BH^n_1(\gamma,\epsilon)$ is a full sphere covering with probability $1 - \epsilon$ for any small $\epsilon>0$, it cannot be used to derive deterministic algorithms and even in some scenarios the feasibility may be questioned as we will see in approximating the tensor nuclear norm in Section~\ref{sec:nnorm}. Moreover, to verify whether $\BH^n_1(\gamma,\epsilon)$ covers the sphere or not, the sphere coverage verification, is NP-hard~\cite{PPL12}. Therefore, explicit and deterministic constructions of hitting sets in $\BT\left(n, \Omega\left(\sqrt{\frac{\ln n}{n}}\right), O(n^\alpha)\right)$ are important. To get this job done, let us first look at two types $\tau$-hitting sets with larger $\tau$.

\subsection{An $\Omega\big(1/{\sqrt{\ln n}}\big)$-hitting set}\label{sec:h2}

As the hitting ratio $\tau$ goes beyond $\Omega\left(\sqrt{\frac{\ln n}{n}}\right)$, we have to give up the polynomiality of $n$. Let us consider
\begin{equation*} 
\BH^n_2:=\left\{\frac{\bz}{\|\bz\|}\in\BS^n: \bz\in\{-1,0,1\}^n,\,\|\bz\|\neq 0\right\}.
\end{equation*}
It is obvious that $|\BH^n_2|<3^n$. We need to work out how large $\tau$ is for this $\tau$-hitting set, essentially Theorem~\ref{thm:h2} below. Interestingly, some results in matroid theory will be used in the proof.


To begin with, let $\BI:= \{1, 2, \dots, n\}$ and its power set $2^{\BI}:= \{\BD: \BD \subseteq \BI \}$. For any $\BD\in 2^{\BI}$, define
\[
\BY^n_\BD:=\left\{\by\in\R^n : y_i \in\{-1,1\} \mbox{ for } i \in \BD \mbox{ and } y_i = 0 \mbox{ for } i \in \BI\setminus\BD\right\},
\]
and denote $\BY^n =\bigcup_{\BD \in 2^{\BI}\setminus\{\emptyset\}} \BY^n_{\BD}$. It is easy to see that $\BH^n_2=\left\{\frac{\by}{\|\by\|}:\by\in\BY^n\right\}$. Our goal is to establish a lower bound of
$\min_{\bx \in \BS^n}\max_{\bz \in \BH^n_2} \bx^{\T} \bz$.
\begin{theorem}\label{thm:h2}
It holds that
\begin{equation}\label{eq:h2thm}
\min_{\bx \in \BS^n}\max_{\bz \in \BH^n_2} \bx^{\T} \bz=\min_{\bx \in \BS^n}\max_{\by \in \BY^n} \frac{\bx^{\T} \by}{\|\by\|} = \min_{\bx \in \BS^n}\max_{\BD \in 2^{\BI}\setminus \{\emptyset\}}\max_{\by \in \BY^n_{\BD}} \frac{\bx^{\T} \by}{\|\by\|} =\min_{\bx \in \BS^n}\max_{\BD \in 2^{\BI}\setminus \{\emptyset\}} \sum_{i \in \BD} \frac{|x_i|} {\sqrt{|\BD |}} \ge \frac{2}{\sqrt{\ln n+5 }}.
\end{equation}
\end{theorem}
It is straightforward to verify all the equalities in~\eqref{eq:h2thm}. To show the inequality, let us consider the following optimization problem
\[
\max\left\{\|\bx\|^2 : \sum_{i \in \BD} |x_i| \le \alpha\sqrt{|\BD|} \mbox{ for all }  \BD \in {2^{\BI}\setminus\{\emptyset\}} \right\},
\]
where $\alpha\ge0$ is a given constant. This is equivalent to
\begin{equation}\label{eq:polymatroid}
\max\left\{\|\bx\|^2 : \bx\ge{\bf0},\,\sum_{i \in \BD} x_i \le \alpha\sqrt{|\BD|} \mbox{ for all }  \BD \in2^{\BI} \right\},
\end{equation}
which is to maximize a strictly convex quadratic function over a polyhedron
\[\BX^n =\left\{ \bx \in \R^n_{+} : \sum_{i \in \BD} x_i \le \alpha\sqrt{|\BD|}  \mbox{ for all } \BD \in {2^{\BI}}\right\}.\]
Therefore, the optimal solution of~\eqref{eq:polymatroid} must be obtained at some extreme points of $\BX^n$. To compute the optimal value, we now characterize extreme optimal points of~\eqref{eq:polymatroid}. We need the following two technical results for the preparation.

\begin{lemma}\label{thm:polymatroid}
$g: 2^{\BI} \to \R$ where $g(\BD) = \alpha \sqrt{|\BD|}$ with $\alpha >0$, then $\BX^n$ is a polymatroid with respect to the function $g$ and the index set $\BI$.
\end{lemma}
\begin{proof}
It suffices to show that $g$ is a rank function, i.e., normalized, nondecreasing and submodular. Obviously, $g(\emptyset)=0$ and $g(\BD_1) = {\color{black}\alpha \sqrt{|\BD_1|} \le \alpha \sqrt{|\BD_2|} }= g(\BD_2)$ whenever $\BD_1 \subseteq \BD_2 \subseteq \BI$. It remains to show the submodularity
\[
g(\BD_1 \cup \BD_2) + g(\BD_1 \cap \BD_2) \le g(\BD_1) + g(\BD_2)\quad \forall\;\BD_1, \BD_2 \subseteq \BI.
\]
If we let $|\BD_1| = a$, $|\BD_2\setminus \BD_1| = b$, and $|\BD_1 \cap \BD_2|=c$, then the above inequality is equivalent to
\[
\sqrt{a+b}+ \sqrt{c} \le \sqrt{a} + \sqrt{b+c}\quad \forall\;a\ge c\ge0, b\ge0.
\]
This is actually implied by
\[
\sqrt{a+b}-\sqrt{a} = \frac{b}{\sqrt{a+b}+\sqrt{a}}  \le  \frac{b}{\sqrt{b+c} +\sqrt{c}} = \sqrt{b+c} -\sqrt{c}.
\]
\end{proof}

The next result is well known regarding an optimal solution of maximizing a linear function over a polymatroid; see e.g.,~\cite{E03}.
\begin{lemma}\label{thm:matroid-greedy}
Consider the linear program
\begin{equation*}
\max\left\{\ba^{\T}\bx : \bx\ge{\bf0},\,\sum_{i\in\BD} x_i \le g(\BD) \mbox{ for all }  \BD \in 2^{\BI} \right\}
\end{equation*}
where $\ba \in \R^n_+$ and $g$ is a rank function. Let $\pi = (\pi_1,\pi_2,\dots, \pi_n)$ be a permutation of $\BI$ with $a_{\pi_1} \ge a_{\pi_2} \ge \dots \ge a_{\pi_n}\ge0$. An optimal solution $\bx$ to the linear program can be obtained by letting
\[
x_{\pi_i}=\left\{
\begin{array}{ll}
g(\{ \pi_i \}) & i=1\\
g(\{ \pi_1,\dots, \pi_i \}) - g(\{ \pi_1,\dots, \pi_{i-1} \}) & i=2,\dots,n.
\end{array}
\right.
\]
\end{lemma}

We can now characterize extreme optimal points and upper bound the optimal value of~\eqref{eq:polymatroid}.
\begin{proposition}\label{thm:matroid-value-bound}
The optimal value of~\eqref{eq:polymatroid} is no more than $\left(\frac{\ln n+5}{4} \right) \alpha^2 $.
\end{proposition}
\begin{proof}
Denote $\bz$ to be an optimal solution of~\eqref{eq:polymatroid}. 
In fact, $\bz$ is the unique optimal solution to the linear program
\begin{equation}\label{prob:LP-matroid}
    \max_{\bx\in\BX^n} \bz^{\T}\bx.
\end{equation}
If this is not true, we then have another $\by \in \BX^n$ with $\by\neq\bz$ and ${\bz}^{\T} \by \ge \bz^{\T}\bz= \|\bz\|^2$. This implies that $\|\by\|^2 >\|\bz\|^2$, invalidating the optimality of $\bz$ to~\eqref{eq:polymatroid}.

Applying Lemma~\ref{thm:polymatroid} and Lemma~\ref{thm:matroid-greedy} to~\eqref{prob:LP-matroid} with $g(\BD) = \alpha\sqrt{|\BD|}$ and $\ba = \bz \in \R^n_+$ and choosing a permutation $\pi$ with $z_{\pi_1} \ge z_{\pi_2} \ge \dots \ge z_{\pi_n}\ge0$, one has
\[
z_{\pi_i}=\left\{
\begin{array}{ll}
\alpha & i=1\\
\alpha\left(\sqrt{i}-\sqrt{i-1}\right) & i=2,\dots,n.
\end{array}
\right.
\]
As a consequence,
\begin{eqnarray*}
\frac{\|\bz\|^2}{\alpha^2} =1+\sum_{i=2}^n\left(\sqrt{i}-\sqrt{i-1}\right)^2
=1+\sum_{i=2}^n\left(\frac{1}{\sqrt{i}+\sqrt{i-1}}\right)^2
\le 1 + \sum_{i=2}^n\frac{1}{4(i-1)}
 \le\frac{\ln n+5}{4},
\end{eqnarray*}
which shows that the optimal value of~\eqref{eq:polymatroid}, $\|\bz\|^2$, is upper bounded by $\left(\frac{\ln n+5}{4} \right)\alpha^2$.
\end{proof}

We are ready to finish the final piece, i.e., to show the inequality in~\eqref{eq:h2thm}. If this is not ture, then there is an $\bx \in \BS^n$ such that
$$\max_{\BD \in {2^{\BI}\setminus\{\emptyset\}}} \sum_{i \in \BD} \frac{|x_i|} {\sqrt{|\BD |}}\le \frac{2\beta }{\sqrt{\ln n+5 }} \mbox{ with } 0<\beta <1.$$
This means that $|\bx|\in\R^n_+$ with $\|\bx\|^2=1$ is a feasible solution to~\eqref{eq:polymatroid} for $\alpha = \frac{2\beta}{\sqrt{\ln n+5 }}$. However, according to Proposition~\ref{thm:matroid-value-bound}, the optimal value of this problem is no more than
\[
\left(\frac{\ln n+5}{4} \right) \alpha^2  = \left(\frac{\ln n+5}{4} \right) \frac{4\beta^2}{\ln n +5} = \beta^2 <1,
\]
giving rise to a contradiction. Finally, we conclude this part as below.
\begin{corollary}\label{thm:h2}
It holds that
\begin{equation*}
\BH^n_2\in\BT\left(n,\frac{2}{\sqrt{\ln n+5}},3^n\right).
\end{equation*}
\end{corollary}

\subsection{$\Omega(1)$-hitting sets}\label{sec:h3}

$\BH^n_2$ is simple and almost close to help the construction of deterministic $\Omega\left(\sqrt{\frac{\ln n}{n}}\right)$-hitting sets, whose story will be revealed in the next subsection. To make this final small but important step, $\Omega(1)$-hitting sets are needed. A finely tuned version of $\BH^n_2$ is in place.

\begin{algorithm}\label{alg:h3}
Given $\BS^n$ and two parameters $\alpha\ge1$ and $\beta\ge\alpha+1$, construct $\BH^n_3(\alpha,\beta)$.

\vspace{-0.2cm}\noindent\hrulefill
\begin{enumerate}
\item Let $m=\left\lceil\log_\beta \alpha n \right\rceil$ and partition $\BI$ into disjoint subsets $\I_1,\I_2,\dots,\I_m$ such that
\begin{equation}\label{eq:partition}
|\I_1|=n-\sum_{k=2}^m|\I_k| \mbox{ and } |\I_k|=\left\lfloor \frac{\alpha n}{\beta^{k-1}} \right\rfloor
\mbox{ for } k=2,3\dots, m.
\end{equation}
\item For any partition $\{\I_1,\I_2,\dots,\I_m\}$ of $\BI$ satisfying~\eqref{eq:partition}, construct a set of vectors
\[\BZ^n(\I_1,\I_2,\dots,\I_m)=\left\{\bz\in\R^n: z_i\in\left\{\pm1,\pm \beta^{\frac{k-1}{2}}\right\} \mbox{ if } i\in\I_k \mbox{ for } k=1,2,\dots,m \right\}.\]
\item Put all $\BZ^n(\I_1,\I_2,\dots,\I_m)$'s together to form\[\BZ^n=\bigcup_{\{\I_1,\I_2,\dots,\I_m\}\mbox{ is a partition of $\BI$ satisfying~\eqref{eq:partition}}}\BZ^n(\I_1,\I_2,\dots,\I_m).\]
\item Project the vectors in $\BZ^n$ onto the unit sphere, i.e.,
$\BH^n_3(\alpha,\beta):=\left\{\frac{\bz}{\|\bz\|}\in\BS^n: \bz\in \BZ^n\right\}.$
\end{enumerate}
\vspace{-0.2cm}\hrulefill
\end{algorithm}

We see from the first step of Algorithm~\ref{alg:h3} that
\[
\sum_{k=2}^m |\I_k| \le \sum_{k=2}^m \frac{\alpha n}{\beta^{k-1}} =\frac{\alpha n}{\beta-1}-\frac{\alpha n}{\beta^{m-1}(\beta-1)}\le\frac{\alpha n}{\beta-1},
\]
implying that
\begin{equation}\label{eq:i1}
|\I_1|=|\BI|-\sum_{k=2}^m |\I_k|\ge n-\frac{\alpha n}{\beta-1}=\left(1-\frac{\alpha}{\beta-1}\right)n \ge0.
\end{equation}
Therefore, the feasibility of the construction is guaranteed. On the other hand, as $m=\left\lceil\log_\beta \alpha n \right\rceil$, we have
\[
\sum_{k=2}^m (|\I_k|+1)\ge \sum_{k=2}^m \frac{\alpha n}{\beta^{k-1}}
= \frac{\alpha n}{\beta-1}-\frac{\alpha n}{\beta^{m-1}(\beta-1)}
\ge \frac{\alpha n}{\beta-1}-\frac{\beta}{\beta-1} \ge\frac{\alpha n}{\beta-1} -2,
\]
implying that
\begin{equation}\label{eq:i2}
|\I_1|=n-\sum_{k=2}^m |\I_k| = n+m-1 -\sum_{k=2}^m (|\I_k|+1) \le n+\log_\beta \alpha n-\frac{\alpha n}{\beta-1} +2
\le \left(1-\frac{\alpha}{\beta-1}\right)n +\log_\beta n+3.
\end{equation}



\begin{theorem}
For any $\bx\in\BS^n$, there exists $\bz\in\BH^n_3(\alpha,\beta)$ such that $\bz^{\T}\bx\ge \frac{\alpha-1}{\sqrt{\alpha\beta(\alpha+1)}}$.
\end{theorem}
\begin{proof}
For any given $\|\bx\| = 1$, define the index sets
\begin{align}
\BD_0(\bx)&=\left\{ i \in \BI : |x_i| \le \frac{1}{\sqrt{\alpha n}} \right\} \nonumber \\
\BD_k(\bx)&=\left\{ i \in \BI : \sqrt{\frac{\beta^{k-1}}{\alpha n}}  < |x_i| \le \sqrt{\frac{\beta^k}{\alpha n}} \right\}\quad k=1,2, \dots, m. \label{eq:ekx}
\end{align}
For any entry $x_i$ of $\bx$, $|x_i|\le 1\le \sqrt{\frac{\beta^m}{\alpha n}}$, and so $\{\BD_0(\bx),\BD_1(\bx),\dots,\BD_m(\bx)\}$ is a partition of $\BI$.

We first estimate $|\BD_k(\bx)|$ for $k\ge2$. {\color{black}It is} obvious that for $k\ge2$,
\[
\frac{\beta^{k-1}}{\alpha n}  |\BD_k(\bx)|= \sum_{i \in \BD_k(\bx)}\frac{\beta^{k-1}}{\alpha n} < \sum_{i \in \BD_k(\bx)}|x_i|^2 \le 1.
\]
This implies that $|\BD_k(\bx)| < \frac{\alpha n}{\beta^{k-1}}$, i.e., $|\BD_k(\bx)| \le\left\lfloor \frac{\alpha n}{\beta^{k-1}}\right\rfloor$ for $k=2,3, \dots, m$. Hence, there exists a partition $\{\I_1,\I_2,\dots,\I_m\}$ of $\BI$ satisfying~\eqref{eq:partition} such that $\BD_k(\bx)\subseteq \I_k$ for $k=2,3, \dots, m$. Furthermore, we may find a vector $\bz\in \BZ^n(\I_1,\I_2,\dots,\I_m)$ such that
\[
z_i = \left\{
\begin{array}{ll}
\sign(x_i) & i \in {\color{black}\BD_0(\bx)} \\
\sign(x_i)\beta^{\frac{k-1}{2}}  & i\in \BD_k(\bx) \mbox{ for } {\color{black}k=1,2,\dots,m,}
\end{array}
\right.
\]
where the $\sign$ function takes $1$ for nonnegative reals and $-1$ for negative reals.

In the following, we shall estimate $\bz^{\T}\bx$ and $\|\bz\|$. First of all,
\[
\sum_{i \in \BD_0(\bx)}{x_i}^2 \le \sum_{i \in \BD_0(\bx)}\frac{1}{\alpha n} \le \frac{1}{\alpha} \mbox{ and } \sum_{k=1}^{m} \sum_{i \in \BD_k(\bx)} {x_i}^2 = \sum_{i\in\BI} {x_i}^2 -  \sum_{i \in \BD_0(\bx)}{x_i}^2 \ge 1-\frac{1}{\alpha}.
\]

Next, we have
\[
\sum_{i \in \BD_0(\bx)}{z_i}^2 = |\BD_0(\bx)| \le n  \mbox{ and }
\sum_{k=1}^{m} \sum_{i \in \BD_k(\bx)} {z_i}^2
=\sum_{k=1}^{m} \sum_{i \in \BD_k(\bx)} \beta^{k-1}
< \sum_{k=1}^{m} \sum_{i \in \BD_k(\bx)} \alpha n{x_i}^2 \le \alpha n.
\]
Summing the above two inequalities would give us $\|\bz\|^2\le (\alpha+1)n$ since $\{\BD_0(\bx),\BD_1(\bx),\dots,\BD_m(\bx)\}$ is a partition of $\BI$.

Lastly, as $\sign(x_i)=\sign(z_i)$ for every $i\in\BI$,
\[
\bz^{\T} \bx \ge \sum_{k=1}^{m} \sum_{i \in \BD_k(\bx)} \beta^{\frac{k-1}{2}} |x_i|
\ge \sum_{k=1}^{m} \sum_{i \in \BD_k(\bx)} \sqrt{\frac{\alpha n}{\beta}}|x_i|^2
=  \sqrt{\frac{\alpha n}{\beta}} \sum_{k=1}^{m} \sum_{i \in \BD_k(\bx)} |x_i|^2
\ge \sqrt{\frac{\alpha n}{\beta}} \left(1-\frac{1}{\alpha}\right),
\]
where the second inequality is due to the upper bound in~\eqref{eq:ekx}.

To conclude, we find $\frac{\bz}{\|\bz\|}\in\BH^n_3(\alpha,\beta)$ such that
\[
\bx^{\T} \frac{\bz}{\|\bz\|} \ge \sqrt{\frac{\alpha n}{\beta}} \left(1-\frac{1}{\alpha}\right) \cdot\frac{1}{\sqrt{(\alpha+1)n}} = \frac{\alpha-1}{\sqrt{\alpha\beta(\alpha+1)}},
\]
proving the desired inequality.
\end{proof}

We also need to estimate the cardinality of $\BH^n_3(\alpha,\beta)$. To simplify the display, we now replace $\beta$ by $\gamma+1$ where $\gamma\ge\alpha$ in the rest of this subsection.
\begin{proposition}
Given two constants $1\le\alpha\le\gamma$, one has
\begin{equation} \label{eq:t}
\left|\BH^n_3(\alpha,\gamma+1)\right|\le \left( 2^{\frac{\gamma+\alpha}{\gamma}}\alpha^{-\frac{\alpha}{\gamma}} (\gamma+1)^{\frac{\alpha(\gamma+1)}{\gamma^2}} \left(\frac{\gamma}{\gamma-\alpha}\right)^{\frac{\gamma-\alpha}{\gamma}}+o(1)\right)^n.
\end{equation}
\end{proposition}
\begin{proof}
For any given partition $\{\I_1,\I_2,\dots,\I_m\}$ of $\BI$ satisfying~\eqref{eq:partition}, $x_i$ can take 2 values if $i\in\I_1$ and 4 values if $i\in\I_k$ for $k\ge2$. By considering the number of such partitions and possible overlaps after projecting on to $\BS^n$, one has
\[
\left|\BH^n_3(\alpha,\gamma+1)\right| \le \frac{n!}{\prod_{k=1}^m|\I_k|!}2^{|\I_1|}\prod_{k=2}^m4^{|\I_k|}.
\]
By~\eqref{eq:i1}, we have
\[
2^{|\I_1|}\prod_{k=2}^m4^{|\I_k|}  = 2^{|\I_1|}4^{\sum_{k=2}^m|\I_k|}= 2^{-|\I_1|}4^{\sum_{k=1}^m|\I_k|} \le 2^{-(1-\frac{\alpha}{\gamma})n} 4^n = 2^{(1+\frac{\alpha}{\gamma})n}.
\]
It remains to estimate $\frac{n!}{\prod_{k=1}^m|\I_k|!}$ based on the
followings from~\eqref{eq:partition},~\eqref{eq:i1} and~\eqref{eq:i2} with $\beta = \gamma+1$,
\[
{\color{black}\eta_1}n \le |\I_1|\le {\color{black}\eta_1}n +\ln_{\gamma+1} n+3
\mbox{ and } |\I_k|=\left\lfloor {\color{black}\eta_k}n \right\rfloor \mbox{ for }k=2,3\dots,m,
\]
where $\eta_1=1-\frac{\alpha}{\gamma}$ and $\eta_k=\frac{\alpha}{(\gamma+1)^{k-1}}$ for $k=2,3\dots,m$, as well as $\sum_{k=1}^m|\I_k|=n$, $m=\left\lceil\log_{\gamma+1} \alpha n \right\rceil$ and $1\le\alpha\le\gamma$.

We first notice that
$$
  \left|n-\sum_{k=1}^m\eta_k n\right|=\left|\sum_{k=1}^m|\I_k|- \sum_{k=1}^m\eta_k n \right|
\le \left||\I_1| - \eta_1 n \right| + \sum_{k=2}^m\left||\I_k| - \eta_k n \right|\le O(\ln n)+m-1=O(\ln n).
$$
We then consider the function $f(x)=x \ln x - x$, which is increasing and convex over $[1,\infty)$ since $f'(x)=\ln x$ and $f''(x)=\frac{1}{x}$. Therefore, for any $x\ge1$ and $y>0$,
$\frac{f(x+y)-f(x)}{y}\le f'(x+y)=\ln(x+y)$, implying that $f(x+y)-f(x)\le y \ln (x+y)$. Applying this fact to $|\I_k|$ for $k=1,2,\dots,m$, we obtain
\[
f(|\I_1|)-f(\eta_1n) = O(\ln^2 n) \mbox{ and } f(\eta_kn) - f(|\I_k|)= O(\ln n) \mbox{ for } k=2,3,\dots,m.
\]
These, together with the Stirling approximation $\ln(n!)=f(n)+O(\ln n)$, lead to
\[
\ln(|\I_1|!)=f(\eta_1n)+ O(\ln^2 n)
\mbox{ and }
\ln(|\I_k|!)=f(\eta_kn) + O(\ln n) \mbox{ for } k=2,3,\dots,m.
\]

We are ready to estimate $\frac{n!}{\prod_{k=1}^m|\I_k|!}$ within a deviation of $o(n)$ as follows:
\begin{align*}
\ln \frac{n!}{\prod_{k=1}^m|\I_k|!} &= \ln (n!) - \sum_{k=1}^{m} \ln (|\I_k|!)\\	
&=   n\ln n - n - \sum_{k=1}^{m} \eta_kn \ln (\eta_kn) + \sum_{k=1}^{m} \eta_kn + o(n)\\
&=   \left(n-  \sum_{k=1}^{m} \eta_kn\right)(\ln n-1) - \sum_{k=2}^{m} \eta_kn \ln \eta_k - \eta_1n \ln \eta_1 + o(n)\\
&=   - n \sum_{k=2}^{m} \frac{\alpha}{(\gamma+1)^{k-1}} \ln \frac{\alpha}{(\gamma+1)^{k-1}} -n \frac{\gamma-\alpha}{\gamma} \ln \frac{\gamma-\alpha}{\gamma}  + o(n)\\
&=  - n\sum_{k=2}^{m} \frac{\alpha \ln \alpha}{(\gamma +1)^{k-1}} + n\sum_{k=2}^{m}  \frac{\alpha(k-1)}{(\gamma +1)^{k-1}}\ln (\gamma +1) +\left( \frac{\gamma - \alpha}{\gamma} \ln \frac{\gamma}{\gamma - \alpha}\right) n + o(n)\\	
&= -\left(\frac{\alpha}{\gamma}\ln\alpha \right) n+ \left(\frac{\alpha(\gamma+1)}{\gamma^2}\ln(\gamma+1) \right)n + \left( \frac{\gamma - \alpha}{\gamma} \ln \frac{\gamma}{\gamma - \alpha}\right) n  + o(n),
\end{align*}
where the last equality is due to the fact that $m=\left\lceil\log_{\gamma+1} \alpha n \right\rceil$ implies
\begin{align*}
\sum_{k=2}^{m} \frac{1}{(\gamma +1)^{k-1}} &= \frac{1}{\gamma} -\frac{1}{\gamma(\gamma+1)^{m-2}}=\frac{1}{\gamma}+o(1)\\
\sum_{k=2}^{m} \frac{k-1}{(\gamma +1)^{k-1}} &= \frac{\gamma+1}{\gamma^2} -\frac{m}{\gamma^2(\gamma+1)^{m-2}} + \frac{m-1}{\gamma^2(\gamma+1)^{m-1}} =\frac{\gamma+1}{\gamma^2}+o(1).
\end{align*}

Finally, by combining the upper bound of $2^{|\I_1|}\prod_{k=2}^m4^{|\I_k|}$, we have $|\BH^n_3(\alpha,\gamma+1)|\le t^n$ where
\[
t = 2^{\frac{\gamma+\alpha}{\gamma}}\alpha^{-\frac{\alpha}{\gamma}} (\gamma+1)^{\frac{\alpha(\gamma+1)}{\gamma^2}} \left(\frac{\gamma}{\gamma-\alpha}\right)^{\frac{\gamma-\alpha}{\gamma}}e^{o(1)}= 2^{\frac{\gamma+\alpha}{\gamma}}\alpha^{-\frac{\alpha}{\gamma}} (\gamma+1)^{\frac{\alpha(\gamma+1)}{\gamma^2}} \left(\frac{\gamma}{\gamma-\alpha}\right)^{\frac{\gamma-\alpha}{\gamma}}+o(1).
\]
\end{proof}

In a {\color{black}nutshell}, we have the following.
\begin{corollary}\label{thm:h3col}
For any given $1\le\alpha\le\gamma$,
\begin{equation}\label{eq:h3}
\BH^n_3(\alpha,\gamma+1) \in \BT\left(n, \frac{\alpha-1}{\sqrt{\alpha(\alpha+1)(\gamma+1)}}, \left( 2^{\frac{\gamma+\alpha}{\gamma}}\alpha^{-\frac{\alpha}{\gamma}} (\gamma+1)^{\frac{\alpha(\gamma+1)}{\gamma^2}} \left(\frac{\gamma}{\gamma-\alpha}\right)^{\frac{\gamma-\alpha}{\gamma}}+o(1) \right)^n \right).
\end{equation}
\end{corollary}

For a fixed $\alpha\ge1$, the largest $\frac{\alpha-1}{\sqrt{\alpha(\alpha+1)(\gamma+1)}}$ is $\frac{\alpha-1}{(\alpha+1)\sqrt{\alpha}}$, achieved when $\gamma=\alpha$. Correspondingly,
\[ 
|\BH^n_3(\alpha,\alpha+1)|\le \left(4\alpha^{-1}(\alpha+1)^{\frac{\alpha+1}{\alpha}}+o(1)\right)^n
\le (4+\epsilon)^n \mbox{ for some large $\alpha$.}
\]
If we maximize $\frac{\alpha-1}{(\alpha+1)\sqrt{\alpha}}$ to achieve the best $\tau$ for the $\tau$-hitting set, this is $\sqrt{\frac{1}{2}(5\sqrt{5}-11)} \approx 0.30028$ when $\alpha=2+\sqrt{5} \approx 4.236$. For reference, we list the $\tau$ and the cardinality for some $\BH^n_3(\alpha,\alpha+1)$ in Table~\ref{table:h3}.
\begin{table}[h]
\centering
\small\begin{tabular}{|l|p{0.4in}p{0.4in}p{0.4in}p{0.4in}p{0.4in}p{0.4in}p{0.4in}|}
\hline
$\alpha$ & 17.42 & 7.64 & 4.75 & $4.24$ & 4.00 & 3.00 & 2.00 \\
\hline
$\tau$ for the $\tau$-hitting set & 0.213 & 0.278 & 0.299 & 0.30028 & 0.300 & 0.288 & 0.235 \\
$t$ for $|\BH^n_3(\alpha,\alpha+1)|\le t^n$ & 5.00 & 6.00 & 7.00 & 7.31 & 7.48 & 8.47 & 10.40\\
\hline
\end{tabular}
\caption{Properties of $\BH^n_3(\alpha,\alpha+1)$ for some $\alpha$.}\label{table:h3}
\end{table}

If we are interested to minimize the upper bound of $|\BH^n_3(\alpha,\gamma+1)|$ in~\eqref{eq:t}, then by fixing $\alpha$ and choosing $\gamma$ sufficiently large, the bound can even be $(2+\epsilon)^n$. However, this makes sense only if $\gamma\ll n$. Moreover, the corresponding $\tau=\frac{\alpha-1}{\sqrt{\alpha(\alpha+1)(\gamma+1)}}$ will decrease quickly as $\gamma$ goes large.

\subsection{$\Omega\big(\sqrt{\ln n/n}\big)$-hitting sets}\label{sec:h4}

The hitting sets in Sections~\ref{sec:h2} and~\ref{sec:h3} can be used to construct new hitting sets which in fact derandomize the constructions in Section~\ref{sec:h1}. Recall that $\BE^n=\{\e_1,\e_2,\dots,\e_n\}$ is the standard basis of $\R^n$, $\boxtimes$ denotes the Kronecker product, and $\vee$ denotes vector appending.
\begin{lemma}\label{thm:kron}
If a hitting set $\BH^{n_1}\in\BT(n_1,\tau,m)$ with $\tau\ge0$, then $\BE^{n_2}\boxtimes\BH^{n_1}\in\BT\left(n_1n_2, \frac{\tau}{\sqrt{n_2}},mn_2\right)$.
\end{lemma}
\begin{proof}
First, for any $\e_i\in\BE^{n_2}$ and $\bz\in\BH^{n_1}$, one has $\|\e_i\boxtimes\bz\|=\|\e_i\|\cdot\|\bz\|=1$. Thus, $\BE^{n_2}\boxtimes\BH^{n_1}\subseteq\BS^{n_1n_2}$. For any $\bx\in\BS^{n_1n_2}$, let
$\bx=\bx_1\vee\bx_2\vee\dots\vee\bx_{n_2}$
where $\bx_k\in\R^{n_1}$ for $k=1,2,\dots,n_2$. Since $\sum_{k=1}^{n_2}\|\bx_k\|^2=\|\bx\|^2=1$, there exists an $\bx_i$, such that $\|\bx_i\|^2\ge\frac{1}{n_2}$.

Observing that $\frac{\bx_i}{\|\bx_i\|}\in\BS^{n_1}$, there exists $\by\in\BH^{n_1}$ such that $\by^{\T}\frac{\bx_i}{\|\bx_i\|}\ge\tau$. Therefore, we have $\e_i\boxtimes\by\in \BE^{n_2}\boxtimes\BH^{n_1}$ satisfying
\[
(\e_i\boxtimes\by)^{\T}\bx=\by^{\T}\bx_i\ge\tau\|\bx_i\|\ge \frac{\tau}{\sqrt{n_2}}.
\]

Finally, by noticing possible overlaps, one has $|\BE^{n_2}\boxtimes\BH^{n_1}|\le |\BE^{n_2}|\cdot|\BH^{n_1}|\le n_2m$.
\end{proof}

\begin{lemma}\label{thm:append}
If two hitting sets $\BH^{n_1}\in\BT(n_1,\tau_1,m_1)$ and $\BH^{n_2}\in\BT(n_2, \tau_2,m_2)$ with $\tau_1,\tau_2>0$, then $$\left(\BH^{n_1}\vee{\bf0}_{n_2}\right) \bigcup\left({\bf0}_{n_1}\vee\BH^{n_2}\right)\in \BT\left(n_1+n_2, \frac{\tau_1\tau_2}{\sqrt{{\tau_1}^2+{\tau_2}^2}},m_1+m_2\right).$$
\end{lemma}
\begin{proof}
For any $\bx\in\BS^{n_1+n_2}$, let $\bx=\bx_1\vee\bx_2$ where $\bx_1\in\R^{n_1}$ and $\bx_2\in\R^{n_2}$. If one of them is a zero vector, say $\bx_1={\bf 0}_{n_1}$, then $\|\bx_2\|=1$. There exists $\by\in\BH^{n_2}$ such that $\by^{\T}\bx_2\ge\tau_2$, and so
\[
\langle {\bf 0}_{n_1}\vee\by, \bx_1\vee\bx_2\rangle = \by^{\T}\bx_2\ge\tau_2\ge \frac{\tau_1\tau_2}{\sqrt{{\tau_1}^2+{\tau_2}^2}}.
\]
If both $\bx_1$ and $\bx_2$ are nonzero, then $\frac{\bx_k}{\|\bx_k\|}\in\BS^{n_k}$ for $k=1,2$. There exist $\by_k\in\BH^{n_k}$ with $\frac{\by_k^{\T}\bx_k}{\|\bx_k\|}\ge\tau_k$ for $k=1,2$. We have
\begin{align*}
\langle \by_1\vee{\bf0}_{n_2}, \bx_1\vee\bx_2 \rangle = \by_1^{\T}\bx_1&\ge \tau_1\|\bx_1\|\\
\langle {\bf0}_{n_1}\vee\by_2, \bx_1\vee\bx_2 \rangle = \by_2^{\T}\bx_2&\ge \tau_2\|\bx_2\|.
\end{align*}
As $\|\bx_1\|^2+\|\bx_2\|^2=\|\bx\|^2=1$, we must have either $\|\bx_1\|\ge\frac{\tau_2}{\sqrt{{\tau_1}^2+{\tau_2}^2}}$ or $\|\bx_2\|\ge\frac{\tau_1}{\sqrt{{\tau_1}^2+{\tau_2}^2}}$. In any case, we have
\[
\max\{\tau_1\|\bx_1\|,\tau_2\|\bx_2\|\}
\ge
\frac{\tau_1\tau_2}{\sqrt{{\tau_1}^2+{\tau_2}^2}},
\]
implying that either $\by_1\vee{\bf0}_{n_2}$ or ${\bf0}_{n_1}\vee\by_2$ is close enough to $\bx$.
\end{proof}

We are ready to construct new hitting sets using $\BH^n_2\in\BT\left(n,\Omega\left(\frac{1}{\sqrt{\ln n}}\right),3^n\right)$ in Section~\ref{sec:h2} and $\BH^n_3\in\BT\left(n,\mu,\nu^n\right)$ with universal constants $\mu,\nu>0$, a handy notation of~\eqref{eq:h3} in Section~\ref{sec:h3}.

\begin{theorem}\label{thm:h4}
Given integer $n\ge2$, let $n_1=\lceil\ln n\rceil$, $n_2=\lfloor\frac{n}{n_1}\rfloor$, and $n_3=n-n_1n_2$. One has
\begin{equation}\label{eq:h4}
\BH^n_4:=\left(\left(\BE^{n_2}\boxtimes\BH^{n_1}_2\right)\vee{\bf0}_{n_3}\right) \bigcup \left({\bf0}_{n_1n_2}\vee\BH^{n_3}_2\right)
\in\BT\left(n,\Omega\left(\sqrt{\frac{\ln n}{n\ln\ln n}}\right),O(n^{1+\ln3})\right).
\end{equation}
\end{theorem}
\begin{proof}
First, as $n_3=n-n_1\lfloor\frac{n}{n_1}\rfloor<n_1$, we have
\begin{equation}\label{eq:n123}
n_1\le\ln n+1,\; n_2\le\frac{n}{\ln n}, \mbox{ and } n_3\le\ln n.
\end{equation}
According to Corollary~\ref{thm:h2} and Lemma~\ref{thm:kron},
\[
\BE^{n_2}\boxtimes\BH^{n_1}_2\in\BT\left(n_1n_2, \Omega\left(\frac{1}{\sqrt{n_2\ln n_1}}\right),n_23^{n_1}\right)
\subseteq \BT\left(n_1n_2, \Omega\left(\sqrt{\frac{\ln n}{n\ln\ln n}}\right),\frac{3n^{1+\ln3}}{\ln n}\right).
\]
Besides, one has
\[
\BH^{n_3}_2\in\BT\left(n_3,\Omega\left(\frac{1}{\sqrt{\ln n_3}}\right),3^{n_3}\right)
\subseteq \BT\left(n_3,\Omega\left(\frac{1}{\sqrt{\ln \ln n}}\right),n^{\ln 3}\right).
\]
Noticing that $\frac{1}{\sqrt{\ln \ln n}}\ge \sqrt{\frac{\ln n}{n\ln\ln n}}$,~\eqref{eq:h4} can be obtained by applying Lemma~\ref{thm:append}.
\end{proof}

Although $\Omega\left(\sqrt{\frac{\ln n}{n\ln\ln n}}\right)$ is slightly lower than $\Omega\left(\sqrt{\frac{\ln n}{n}}\right)$, the construction of $\BH^n_4$ in~\eqref{eq:h4} is very simple (using $\BH^n_2$) and enjoys a low cardinality $O(n^{1+\ln3})\le O(n^{2.1})$. {\color{black} We remark that it is even possible to construct an $\Omega\left(\sqrt{\frac{\ln n}{n\ln\ln n}}\right)$-hitting set with a lower cardinality $O(n^{1.5})$. This can be done by using $\BH^{n_1}_0(m)$ with $m=\lceil\sqrt{\ln n}\rceil$ and $n_1=\lceil\frac{\ln n}{\ln\ln n}\rceil$ in place of $\BH^{n_1}_2$ in constructing $\BH^n_4$ in Theorem~\ref{thm:h4}}. We leave the details to interested readers. In order to remove the ${\color{black}\frac{1}{\sqrt{\ln\ln n}}}$ factor, we need to make use of $\BH^n_3(\alpha,\beta)$.

\begin{theorem}\label{thm:h5}
Given integer $n\ge2$, let $n_1=\lceil\ln n\rceil$, $n_2=\lfloor\frac{n}{n_1}\rfloor$, and $n_3=n-n_1n_2$. By choosing any $\BH^n_3(\alpha,\beta)\in\BT\left(n,\mu,\nu^n\right)$ in~\eqref{eq:h3} with $\alpha\ge1$ and $\beta\ge\alpha+1$, one has
\[
\BH^n_5(\alpha,\beta):=\left(\left(\BE^{n_2}\boxtimes\BH^{n_1}_3(\alpha,\beta)\right)\vee{\bf0}_{n_3}\right) \bigcup \left({\bf0}_{n_1n_2}\vee\BH^{n_3}_3(\alpha,\beta)\right)
\in\BT\left(n,\mu\sqrt{\frac{\ln n}{n+\ln n}},O(n^{1+\ln\nu})\right).
\]
\end{theorem}
\begin{proof}
The proof is similar to that of Theorem~\ref{thm:h4} by noticing~\eqref{eq:n123} and applying Lemma~\ref{thm:kron} and Lemma~\ref{thm:append}. We only need to carry out the calculations.

The $\tau$ for the $\tau$-hitting set $\BE^{n_2}\boxtimes\BH^{n_1}_3(\alpha,\beta)$ is $\frac{\mu}{\sqrt{n_2}}$ and that for $\BH^{n_3}_3(\alpha,\beta)$ is $\mu$. By Lemma~\ref{thm:append}, the $\tau$ for $\BH^n_5(\alpha,\beta)$ is
$
\frac{\frac{\mu}{\sqrt{n_2}}\cdot\mu}{\sqrt{\frac{\mu^2}{n_2}+\mu^2}} = \frac{\mu}{\sqrt{n_2+1}} \ge \mu\sqrt{\frac{\ln n}{n+\ln n}}.
$

For the cardinality, one has $|\BE^{n_2}\boxtimes\BH^{n_1}_3(\alpha,\beta)|\le n_2\nu^{n_1}\le \frac{\nu n^{1+\ln \nu}}{\ln n}$ and $|\BH^{n_3}_3(\alpha,\beta)|\le \nu^{n_3}\le n^{\ln\nu}$. Adding up these two would give $O(n^{1+\ln\nu})$.
\end{proof}

With the cardinality $O(n^{1+\ln\nu})$ in place, it is natural to select the best $\mu$ for $\BH^{n}_3(\alpha,\beta)$ with $\beta = \gamma +1$ in~\eqref{eq:h3}. According to Table~\ref{table:h3}, the largest $\mu=0.30028$ with $\nu=7.31$, obtained when $\alpha=2+\sqrt{5}$ and $\beta=3+\sqrt{5}$. This results a cardinality $O(n^{1+\ln\nu})\le O(n^3)$. To conclude, we have
\begin{equation}\label{eq:h6}
\BH^n_5(2+\sqrt{5},3+\sqrt{5})\in\BT\left(n,0.3\sqrt{\frac{\ln n}{n}},O(n^3)\right).
\end{equation}

{\color{black}
Before concluding this section, we remark that the estimated $\tau$ serves a lower bound and $m$ serves an upper bound for the proposed hitting sets. We evaluate their exact values of one example ($n=6$) by numerical computations, shown in Table~\ref{table:list2} {\color{black} where the Greek letters in estimated values are some unknown constants}. Due to the randomness of $\BH^n_1$, we try ten times for any $m$ and provide corresponding $\tau$ by an interval range.

\begin{table}[h]
\centering
\small\begin{tabular}{|l|cc|cc|}
\hline
Hitting set in $\BS^6$ & Exact $\tau$ & Exact $m$ & Estimated $\tau$  & Estimated $m$   \\
\hline
A regular simplex in $\BS^6$ & 0.167 & 7 & 0.167 & 7 \\
$\BE^6\cup(-\BE^6)$ & 0.408 & 12 & 0.408 & 12\\
$\BH^6_1(\gamma_1,\epsilon_1)$  & [0.331, 0.442] & 27  &  $\omega_1\cdot 0.546$ & $o_1\cdot 6^{\alpha_1}$ \\
$\BH^6_1(\gamma_2,\epsilon_2)$ & [0.521, 0.592] & 60  &  $\omega_2\cdot 0.546$ & $o_2\cdot 6^{\alpha_2}$ \\
$\BH^6_4$ & 0.546 & 27 & $\omega_3\cdot 0.546$ & $o_3\cdot 6^{2.792}$ \\
$\BH^6_5(2+\sqrt{5},3+\sqrt{5})$ & 0.544 & 36  & $\omega_4\cdot 0.546$ & $o_4\cdot 6^{2.989}$ \\
$\BH^6_2$ & 0.835 & 728 & $\omega_5\cdot 0.747$  & $3^6=729$ \\
$\BH^6_3(2+\sqrt{5},3+\sqrt{5})$ & 0.820 & 16896  & $\omega_6\cdot 1.000$ & $7.31^6=152582$\\
\hline
\end{tabular}
\caption{Exact and theoretical estimates of $\tau$ and $m$ for hitting sets in $\BS^6$}\label{table:list2}
\end{table}

For our main constructions of $\Omega\left(\sqrt{\frac{\ln n}{n}}\right)$-hitting sets, $\BH_1^n$, $\BH_4^n$ and $\BH_5^n$, comparisons of $\tau$ and $m$ for a few small $n$'s are shown in Table~\ref{table:list3}. We observe that random hitting sets will outperform deterministic ones when $n$ increases although they are worse when $n$ is small.
\begin{table}[h]
\centering
\small\begin{tabular}{|l|cc|cc|cc|cc|}
\hline
&  \multicolumn{2}{c|}{$n=6$} & \multicolumn{2}{c|}{$n=8$} & \multicolumn{2}{c|}{$n=12$}  & \multicolumn{2}{c|}{$n=15$} \\
Hitting set & $\tau$  & $m$& $\tau$ & $m$  & $\tau$  & $m$  & $\tau$ & $m$ \\
\hline
$\BH^n_1$ & [0.331, 0.442] & 24 & [0.272, 0.384] & 32 & [0.368, 0.410] & 104 & [0.320, 0.391] & 130 \\
$\BH^n_1$  & [0.521, 0.592] & 60 & [0.460, 0.502] & 80 & [0.441, 0.484] & 184 & [0.428, 0.452] & 235 \\
$\BH^n_4$ & 0.546 & 24 & 0.485 & 32 & 0.4653 & 104 & 0.431  & 130 \\
$\BH^n_5$ & 0.544 & 36 & 0.489 & 48 & 0.4713 & 256 & 0.433 & 320 \\
\hline
\end{tabular}
\caption{Exact $\tau$ and $m$ for $\Omega\big(\sqrt{\ln n/ n}\big)$-hitting sets in $\BS^n$}
\label{table:list3}
\end{table}
}

\section{Approximating tensor norms}\label{sec:tensor}

In this section we apply explicit constructions for sphere covering, in particular the deterministic $\Omega\left(\sqrt{\frac{\ln n}{n}}\right)$-hitting sets in Section~\ref{sec:h4}, to derive new approximation methods for the tensor spectral norm and nuclear norm. Let us formally define the approximation bound for tensor norms.
\begin{definition}
A tensor norm $\|\bullet\|_\omega$ can be approximated with an approximation bound $\tau\in(0,1]$, if there exists a polynomial-time algorithm that computes a quantity $\omega_\TT$ for any tensor instance $\TT$ in the concerned space, such that $\tau \|\TT\|_{\omega} \le \omega_\TT\le \|\TT\|_{\omega}$.
\end{definition}

Obviously the larger the $\tau$, the better the approximation bound. We consider the tensor space $\R^{n_1\times n_2 \times \dots \times n_d}$ of order $d\ge3$ and assume without loss of generality that $2\le n_1\le n_2 \le\dots\le n_d$.

\subsection{Approximation bound for tensor spectral norm}\label{sec:snorm}

Given a tensor $\TT\in\R^{n_1\times n_2 \times \dots \times n_d}$, let us denote (recall that $\otimes$ stands for the outer product)
\begin{equation}\label{eq:mform}
\TT(\bx_1,\bx_2,\dots,\bx_d) = \left\langle \TT, \bx_1\otimes\bx_2\otimes\dots\otimes\bx_d \right\rangle = \sum_{i_1=1}^{n_1}\sum_{i_2=1}^{n_2}\dots\sum_{i_d=1}^{n_d} t_{i_1i_2\dots i_d} (x_1)_{i_1}(x_2)_{i_2}\dots (x_d)_{i_d}
\end{equation}
to be the multilinear function of vector entries $(\bx_1,\bx_2,\dots,\bx_d)$ where $\bx_k\in\R^{n_k}$ for $k=1,2,\dots,d$. If any vector entry, say $\bx_1$, is missing and replaced by $\bullet$, then $\TT(\bullet,\bx_2,\bx_3,\dots,\bx_d)\in\R^{n_1}$ becomes a vector. Similarly, $\TT(\bullet,\bullet,\bx_3,\bx_4,\dots,\bx_d)\in\R^{n_1\times n_2}$ is a matrix, and so on.
\begin{definition}\label{def:snorm}
For a given tensor $\TT\in\R^{n_1\times n_2\times\dots\times n_d}$, the spectral norm of $\TT$ is defined as
\begin{equation} \label{eq:defsnorm}
\|\TT\|_{\sigma}:=\max\left\{\TT(\bx_1,\bx_2,\dots,\bx_d): \|\bx_k\|=1,\,\bx_k\in\R^{n_k},\, k=1,2,\dots,d\right\}.
\end{equation}
\end{definition}

The tensor spectral norm was proposed by Lim~\cite{L05} in terms of singular values of a tensor. In light of~\eqref{eq:mform}, $\|\TT\|_{\sigma}$ is the maximal value of the Frobenius inner product between $\TT$ and a rank-one tensor whose Frobenius norm is one since $\|\bx_1\otimes\bx_2\otimes\dots\otimes\bx_d\|=\prod_{k=1}^d\|\bx_k\|=1$.

When $d=2$,~\eqref{eq:defsnorm} is reduced to the matrix spectral norm or the largest singular value of the matrix, which can be computed in polynomial time (e.g., via singular value decompositions). He et al.~\cite{HLZ10} showed that~\eqref{eq:defsnorm} is NP-hard when $d\ge3$. They also proposed the first polynomial-time algorithm with a worst-case approximation bound $\left(\prod_{k=1}^{d-2}\frac{1}{n_k}\right)^{\frac{1}{2}}$. The best known approximation bound for the tensor spectral norm is $\Omega\left(\left(\prod_{k=1}^{d-2}\frac{\ln n_k}{n_k}\right)^{\frac{1}{2}}\right)$ by So~\cite{S11}. However, the method in~\cite{S11} relies on the equivalence between convex optimization and membership oracle queries using the ellipsoid method and it is computationally impractical. There is also a simple but randomized algorithm for the same best bound proposed in~\cite{HJLZ14}. Here in this subsection we are able to present the first easily implementable and deterministic algorithm based on sphere covering, with the same approximation bound $\Omega\left(\left(\prod_{k=1}^{d-2}\frac{\ln n_k}{n_k}\right)^{\frac{1}{2}}\right)$. To make an exact bound without involving $\Omega$, we need to use $\BH^n_5(2+\sqrt{5},3+\sqrt{5})\in\BT\left(n,0.3\sqrt{\frac{\ln n}{n}},O(n^3)\right)$ in~\eqref{eq:h6}.

\begin{algorithm}\label{alg:snorm}
Given $\TT\in\R^{n_1\times n_2\times\dots\times n_d}$, find approximate spectral norm of $\TT$.

\vspace{-0.2cm}\noindent\hrulefill
\begin{enumerate}
\item Enumerate $\bz_k\in\BH^{n_k}_5(2+\sqrt{5},3+\sqrt{5})$ for $k=1,2,\dots,d-2$ and solve the resulting matrix spectral norm problem $\max\left\{\TT(\bz_1,\bz_2,\dots,\bz_{d-2},\bx_{d-1},\bx_{d}):\|\bx_{d-1}\|=\|\bx_{d}\|=1 \right\}$ whose optimal solution is denoted by $(\bz_{d-1},\bz_d)$.
\item Compare all the objective values in the first step and output the largest one.
\end{enumerate}
\vspace{-0.2cm}\hrulefill
\end{algorithm}

It is obvious that Algorithm~\ref{alg:snorm} runs in polynomial time as $|\BH^{n_k}_5(2+\sqrt{5},3+\sqrt{5})|=O({n_k}^3)$ and the matrix spectral norm is polynomial-time computable. Moreover, the corresponding approximate solution $(\bz_1,\bz_2,\dots,\bz_{d-2})$ is universal, i.e., $\bz_k\in\BH^{n_k}_5(2+\sqrt{5},3+\sqrt{5})$ is independent of the data $\TT$ for $k=1,2,\dots,d-2$.

\begin{theorem}\label{thm:snorm}
Algorithm~\ref{alg:snorm} is a deterministic polynomial-time algorithm that approximates $\|\TT\|_\sigma$ with a worst-case approximation bound $0.3^{d-2}\left(\prod_{k=1}^{d-2}\frac{\ln n_k}{n_k}\right)^{\frac{1}{2}}$ for any $\TT\in\R^{n_1\times n_2\times\dots\times n_d}$, i.e., $\bz_k\in\BH^{n_k}_5(2+\sqrt{5},3+\sqrt{5})$ for $k=1,2,\dots,d-2$ and $\bz_k\in\BS^{n_k}$ for $k=d-1,d$ can be found such that
\[
0.3^{d-2}\left(\prod_{k=1}^{d-2}\frac{\ln n_k}{n_k}\right)^{\frac{1}{2}} \|\TT\|_\sigma \le \TT(\bz_1,\bz_2,\dots,\bz_d)\le \|\TT\|_\sigma.
\]
\end{theorem}
\begin{proof}
Let us denote $\tau_k=0.3\sqrt{\frac{\ln n_k}{n_k}}$ for $k=1,2,\dots,d-2$. Let $(\by_1,\by_2,\dots,\by_d)$ be an optimal solution of~\eqref{eq:defsnorm}, i.e., $\TT(\by_1,\by_2,\dots,\by_d)=\|\TT\|_\sigma$.
For the vector $\bv_1 = \TT(\bullet,\by_{2},\by_{3},\dots,\by_{d})$, either $\|\bv_1\|=0$ or there exists $\bz_1\in\BH^{n_1}_5(2+\sqrt{5},3+\sqrt{5})$ such that $\bz_1^{\T}\frac{\bv_1}{\|\bv_1\|}\ge\tau_1$. In any case, one has
\[
\TT(\bz_1,\by_{2},\by_{3},\dots,\by_{d}) = \bz_1^{\T} \bv_1 \ge \tau_1 \|\bv_1\|
\ge \tau_1 \by_1^{\T}\bv_1=\tau_1 \TT(\by_1,\by_2,\dots,\by_{d}).
\]
Similarly, for every $k=2,3\dots,d-2$ that are chosen one by one increasingly, there exists $\bz_k\in\BH^{n_k}_5(2+\sqrt{5},3+\sqrt{5})$ such that
\[
\TT(\bz_1,\dots,\bz_{k-1},\bz_k,\by_{k+1},\dots,\by_{d}) = \bz_k^{\T} \bv_k \ge \tau_k \|\bv_k\|
\ge \tau_k \by_k^{\T}\bv_k=\tau_k \TT(\bz_1,\dots,\bz_{k-1},\by_k,\by_{k+1},\dots,\by_{d}),
\]
where $\bv_k=\TT(\bz_1,\dots,\bz_{k-1},\bullet,\by_{k+1},\dots,\by_{d})$. By applying the above inequalities recursively, we obtain
\[
\TT(\bz_1,\bz_2,\dots,\bz_{d-2},\by_{d-1},\by_{d})\ge \left(\prod_{k=1}^{d-2}\tau_k\right) \TT(\by_1,\by_2\dots,\by_{d})  = \left(\prod_{k=1}^{d-2}\tau_k\right) \|\TT\|_\sigma.
\]

The first step of Algorithm~\ref{alg:snorm} must have enumerated this $(\bz_1,\bz_2,\dots,\bz_{d-2})$ and computed corresponding $\bz_{d-1}\in\BS^{n_{d-1}}$ and $\bz_d\in\BS^{n_d}$, such that
\begin{align*}
\TT(\bz_1,\bz_2,\dots,\bz_d)& =\max_{\|\bx_{d-1}\|=\|\bx_{d}\|=1} \TT(\bz_1,\bz_2,\dots,\bz_{d-2},\bx_{d-1},\bx_{d})\\
&\ge   \TT(\bz_1,\bz_2,\dots,\bz_{d-2},\by_{d-1},\by_{d}) \\
&\ge \left(\prod_{k=1}^{d-2}\tau_k\right) \|\TT\|_\sigma.
\end{align*}
Finally, the best one found by the second step must be no less than the above $\TT(\bz_1,\bz_2,\dots,\bz_d)$.
\end{proof}

A closely related problem to the tensor spectral norm~\eqref{eq:defsnorm} is sphere constrained homogenous polynomial optimization $\max\left\{p(\bx):\|\bx\|=1\right\}$ where $p(\bx)$ a homogenous polynomial function of degree $d$. In other words, there is a symmetric (entries are invariant under permutations of indices) tensor $\TT\in\R^{n\times n\times\dots\times n}$ of order $d$ such that $p(\bx)=\TT(\bx,\bx,\dots,\bx)$. This is a widely applicable optimization problem but is also NP-hard when the degree of the polynomial $d\ge3$~\cite{N03}. The current best approximation bound for this problem is $\Omega\left(\left(\frac{\ln n}{n}\right)^{d/2-1}\right)$, obtained by {\color{black} a randomized algorithm~\cite{HJLZ14} or a deterministic but not implementable algorithm~\cite{S11} as it relies on the equivalence between convex optimization and membership oracle queries using the ellipsoid method. In fact, it is not difficult to obtain an easily implementable deterministic algorithm with the same best approximation bound with the help of a polarization formula~\cite[Lemma 1]{HLZ10} below.}
\begin{lemma} 
Let $\TT\in\R^{n\times n\times\dots\times n}$ be a symmetric tensor of order $d$ and $p(\bx)=\TT(\bx,\bx,\dots,\bx)$. If $\xi_i,\xi_2,\dots,\xi_d$ are i.i.d. symmetric Bernoulli random variables (taking values $\pm1$ with equal probability), then
\[
\ex \left[\left(\prod_{i=1}^d\xi_i\right)p\left(\sum_{k=1}^d\xi_k\bx_k\right) \right] =d!\TT(\bx_1,\bx_2,\dots,\bx_d).
\]
\end{lemma}
We only state the results but leave the details to interested readers.
\begin{theorem}
Let $p(\bx)$ be a homogenous polynomial function of dimension $n$ and degree $d\ge3$. If $d$ is odd, then there is a deterministic polynomial-time approximation algorithm which outputs $\bz\in\BS^n$, such that
\[
p(\bz)\ge 0.3^{d-2}d!d^{-d}\left(\frac{\ln n}{n}\right)^{d/2-1} \max_{\|\bx\|=1}p(\bx).
\]
If $d$ is even, then there is a deterministic polynomial-time approximation algorithm which outputs $\bz\in\BS^n$, such that
\[
p(\bz)-\min_{\|\bx\|=1}p(\bx) \ge 0.3^{d-2}d!d^{-d}\left(\frac{\ln n}{n}\right)^{d/2-1} \left(\max_{\|\bx\|=1}p(\bx)-\min_{\|\bx\|=1}p(\bx)\right).
\]
\end{theorem}

\subsection{Approximation bound for tensor nuclear norm}\label{sec:nnorm}

We now study the approximation for the tensor nuclear norm.
\begin{definition}\label{def:nnorm}
For a given tensor $\TT\in\R^{n_1\times n_2\times\dots\times n_d}$, the nuclear norm of $\TT$ is defined as
\begin{equation} \label{eq:ndecomp}
\|\TT\|_*:=\min\left\{\sum_{i=1}^r|\lambda_i| : \TT=\sum_{i=1}^r\lambda_i\, \bx_1^{(i)}\otimes\bx_2^{(i)}\otimes\dots \otimes\bx_d^{(i)}, \|\bx_k^{(i)}\|=1\mbox{ for all $i$ and $k$}, \, r\in\BN \right\}.
\end{equation}
\end{definition}

{\color{black}From~\eqref{eq:ndecomp}, we see that the tensor nuclear norm is the minimum of the sum of Frobenius norms of rank-one tensors in any CP decomposition. A CP decomposition of $\TT$ that attains $\|\TT\|_*$ is called a nuclear decomposition of $\TT$~\cite{FL18}}. When $d=2$, the tensor nuclear norm is reduced to the matrix nuclear norm, which is the sum of all singular values. Similar to the role of matrix nuclear norm used in many matrix rank minimization problems, the tensor nuclear norm is the convex envelope of the tensor rank and is widely used in tensor completions~\cite{GRY11,YZ16}.

The tensor nuclear norm is the dual norm to the tensor spectral norm, and vice versa, whose proof can be found in~\cite{LC14,CL20}.
\begin{lemma} \label{thm:dual}
For given tensors $\TT$ and $\Z$ in a same tensor space, it follows that
\begin{equation}
\|\TT\|_{\sigma}=\max_{\|\Z\|_*\le 1}\langle\TT,\Z\rangle
\mbox{ and }
\|\TT\|_*=\max_{\|\Z\|_{\sigma}\le 1}\langle\TT,\Z\rangle. \label{eq:def2nnorm}
\end{equation}
\end{lemma}

Computing the tensor nuclear norm is also NP-hard when $d\ge3$ showed by Friedland and Lim~\cite{FL18}. In fact, it is much harder than computing the tensor spectral norm. From the definition~\eqref{eq:ndecomp} finding a CP decomposition is not an easy task for a given $r$, and from the dual formulation~\eqref{eq:def2nnorm} checking the feasibility $\|\Z\|_{\sigma}\le 1$ is also NP-hard. Perhaps the only known method is due to Nie~\cite{N17}, which is based on the sum-of-squares relaxation and can only work for symmetric tensors of low dimensions. In terms of polynomial-time approximation bounds, the best bound is $\prod_{k=1}^{d-2}\frac{1}{\sqrt{n_k}}$. There are two methods to achieve this bound, one is via matrix flattenings of the tensor~\cite{H15} and the other is via partitioning the tensor into matrix slices~\cite{L16}. This bound is {\color{black}worse} than the best one for the tensor spectral norm. Let us now bridge the gaps using an idea similar to grid sampling in~\cite{HJL22}.

To better illustrate our main idea, we discuss the details for a tensor $\TT\in\R^{n_1\times n_2\times n_3}$ of order three. According to the dual formulation~\eqref{eq:def2nnorm},
\begin{align}
\|\TT\|_*& = \max \left\{\langle\TT,\Z\rangle: \|\Z\|_{\sigma}\le 1\right\}\nonumber \\
&= \max \left\{\langle\TT,\Z\rangle: \Z(\bx,\by,\bz)\le1 \mbox{ for all } \|\bx\|=\|\by\|=\|\bz\|=1\right\} \nonumber \\
&= \max \left\{\langle\TT,\Z\rangle: \max_{\|\by\|=\|\bz\|=1}\Z(\bx,\by,\bz)\le1 \mbox{ for all } \|\bx\|=1\right\}. \label{eq:sdp0}
\end{align}
Notice that for a given $\bx$, the constraint $\max_{\|\by\|=\|\bz\|=1} \Z(\bx,\by,\bz) \le 1$ is the same to $\|\Z(\bx,\bullet,\bullet)\|_\sigma\le1$ or the largest singular value of the matrix $\Z(\bx,\bullet,\bullet)$ is no more than one. This can be equivalently represented by $I\succeq\Z(\bx, \bullet ,\bullet)\Z(\bx, \bullet ,\bullet)^{\T}$. Here a symmetric matrix $A\succeq O$ {\color{black}where $O$ is a zero matrix} means that $A$ is positive semidefinite and $A\succeq B$ means that $A-B \succeq O$. Applying the Schur complement, we then have
\[
\max_{\|\by\|=\|\bz\|=1} \Z(\bx,\by,\bz) \le 1
\Longleftrightarrow I\succeq\Z(\bx, \bullet ,\bullet)\Z(\bx, \bullet ,\bullet)^{\T} \Longleftrightarrow
\left[\begin{array}{cc}
I & \Z(\bx, \bullet ,\bullet) \\
\Z(\bx, \bullet ,\bullet)^{\T} & I
\end{array} \right] \succeq O.
\]
By combining with~\eqref{eq:sdp0} we obtain an equivalent formulation of the tensor nuclear norm
\begin{equation}\label{eq:sdp2}
\|\TT\|_* = \max \left\{\langle\TT,\Z\rangle: \left[\begin{array}{cc}
I & \Z(\bx, \bullet ,\bullet) \\
\Z(\bx, \bullet ,\bullet)^{\T} & I
\end{array} \right] \succeq O \mbox{ for all } \|\bx\|= 1\right\}.
\end{equation}

Obviously there is no way to enumerate all $\bx$ in $\BS^{n_1}$ in~\eqref{eq:sdp2} but the sphere covering is indeed helpful in this scenario. If we replace $\|\bx\|=1$ with $\bx\in\BH^{n_1}$ for some deterministic $\BH^{n_1}\in\BT(n_1,\tau,O({n_1}^\alpha))$ with some university constant $\alpha$,~\eqref{eq:sdp2} is then relaxed to
\[
\max \left\{\langle\TT,\Z\rangle: \left[\begin{array}{cc}
I & \Z(\bx, \bullet ,\bullet) \\
\Z(\bx, \bullet ,\bullet)^{\T} & I
\end{array} \right] \succeq O \mbox{ for all } \bx\in\BH^{n_1}\right\}.
\]
This becomes a semidefinte program with $O({n_1}^\alpha)$ number of positive semidefinite constraints.

\begin{algorithm} \label{alg:3nnorm}
Given $\TT\in\R^{n_1\times n_2\times n_3}$, find approximate nuclear norm of $\TT$.

\vspace{-0.2cm}\noindent\hrulefill
\begin{enumerate}
\item Choose a $\tau$-hitting set $\BH^{n_1}\in\BT(n_1,\tau,O({n_1}^\alpha))$ and solve the semidefinite program
\begin{equation}\label{eq:sdp3}
u = \max \left\{\langle\TT,\Z\rangle: \left[\begin{array}{cc}
I & \Z(\bx, \bullet ,\bullet) \\
\Z(\bx, \bullet ,\bullet)^{\T} & I
\end{array} \right] \succeq O \mbox{ for all } \bx\in\BH^{n_1}\right\}.
\end{equation}
\item Output $\tau u$.
\end{enumerate}
\vspace{-0.2cm}\hrulefill
\end{algorithm}

\begin{theorem}\label{thm:3nnorm}
For any $\BH^{n_1}\in\BT(n_1,\tau,O({n_1}^\alpha))$, Algorithm~\ref{alg:3nnorm} is a deterministic polynomial-time algorithm that approximates $\|\TT\|_*$ with a worst-case approximation bound $\tau$.
\end{theorem}
\begin{proof}
Denote $\Y$ to be an optimal solution of~\eqref{eq:sdp3}. It is easy to see that~\eqref{eq:sdp3} is a relaxation of the maximization problem~\eqref{eq:sdp2} since $\BH^{n_1}\subseteq\BS^{n_1}$. Therefore, $u=\langle \TT,\Y\rangle\ge\|\TT\|_*$.

For any $\by,\bz$, denote $\bv=\Y(\bullet,\by,\bz)$ and we have either $\|\bv\|=0$ or there exists $\bx\in\BH^{n_1}$ such that $\bx^{\T}\frac{\bv}{\|\bv\|}\ge\tau$, both leading to $\Y(\bx,\by,\bz)=\bx^{\T}\bv\ge\tau\|\bv\|=\tau \|\Y(\bullet,\by,\bz)\|$. Therefore,
\[
\max_{\bx \in \BH^{n_1},\|\by\|=\|\bz\|=1} \Y(\bx,\by,\bz)
\ge \tau
\max_{\|\by\|=\|\bz\|=1} \|\Y(\bullet,\by,\bz)\|
=\tau \max_{\|\bx\|=\|\by\|=\|\bz\|=1}\Y(\bx,\by,\bz)=\tau\|\Y\|_\sigma.
\]
By the feasibility of $\Y$ in~\eqref{eq:sdp3}, $\|\Y(\bx,\bullet,\bullet)\|_\sigma\le1$ for all $\bx\in\BH^{n_1}$, implying that
\[\|\tau\Y\|_\sigma=\tau\|\Y\|_\sigma\le \max_{\bx \in \BH^{n_1},\|\by\|=\|\bz\|=1} \Y(\bx,\by,\bz) =
\max_{\bx \in \BH^{n_1}} \| \Y(\bx,\bullet,\bullet)\|_\sigma \le 1.
\]
This means that $\tau\Y$ is a feasible solution to the dual formulation~\eqref{eq:def2nnorm}, and so
\[
\|\TT\|_*=\max_{\|\Z\|_\sigma\le 1}\langle \TT,\Z\rangle \ge \langle \TT,\tau\Y\rangle = \tau \langle \TT,\Y\rangle=\tau u\ge\tau\|\TT\|_*.
\]
\end{proof}

Compared to Algorithm~\ref{alg:snorm} that requires (possibly large) enumeration and then comparison, Algorithm~\ref{alg:3nnorm} only needs to solve one semidefinite program, albeit the size is large if $\BH^{n_1}$ is large. We emphasize that $\BH^{n_1}$ in Algorithm~\ref{alg:3nnorm} {\color{black} needs to be a deterministic $\tau$-hitting set in order to archive a feasible solution of $\|\Z\|_\sigma\le1$ in~\eqref{eq:def2nnorm} with the desired approximation bound $\tau$ in Theorem~\ref{thm:3nnorm}. Although a randomized hitting set $\BH^{n_1}_1(\gamma,\epsilon)$ can be used in Algorithm~\ref{alg:3nnorm}, it is likely that $\tau\Y$ in the proof of Theorem~\ref{thm:3nnorm} is not feasible to~\eqref{eq:def2nnorm}. However, $\langle \TT,\Y\rangle$ could still be a good upper bound of $\|\TT\|_*$ in this case.} Let us now extend Algorithm~\ref{alg:3nnorm} to a general tensor of order $d$.

\begin{algorithm} \label{alg:nnorm}
Given $\TT\in\R^{n_1\times n_2\times \dots\times n_d}$, find approximate nuclear norm of $\TT$.

\vspace{-0.2cm}\noindent\hrulefill
\begin{enumerate}
\item Choose $\BH^{n_k}\in\BT(n_k,\tau_k,O({n_k}^{\alpha_k}))$ for $k=1,2,\dots,d-2$ and solve the semidefinite program
\[
u = \max \left\{\langle\TT,\Z\rangle: \left[\begin{array}{cc}
I & \Z(\bx_1,\bx_2,\dots,\bx_{d-2}, \bullet ,\bullet) \\
\Z(\bx_1,\bx_2,\dots,\bx_{d-2}, \bullet ,\bullet)^{\T} & I
\end{array} \right] \succeq O \mbox{ for all } \bx_k\in\BH^{n_k}\right\}.
\]
\item Output $u \prod_{k=1}^{d-2}\tau_k$.
\end{enumerate}
\vspace{-0.2cm}\hrulefill
\end{algorithm}

We state the final theorem that obtains an improved approximation bound for the tensor nuclear norm using the hitting set $\BH^n_5(2+\sqrt{5},3+\sqrt{5})$ in~\eqref{eq:h6}. This bound finally matches the current best one for the tensor spectral norm; see Theorem~\ref{thm:snorm}. The proof is similar to that of Theorem~\ref{thm:3nnorm} and is omitted.

\begin{theorem}\label{thm:nnorm}
By choosing $\BH^{n_k}_5(2+\sqrt{5},3+\sqrt{5})$ for $k=1,2,\dots,d-2$, Algorithm~\ref{alg:nnorm} is a deterministic polynomial-time algorithm that approximates $\|\TT\|_*$ with a worst-case approximation bound $0.3^{d-2}\left(\prod_{k=1}^{d-2}\frac{\ln n_k}{n_k}\right)^{\frac{1}{2}}$ for any $\TT\in\R^{n_1\times n_2\times\dots\times n_d}$.
\end{theorem}

\subsection{Numerical performance of approximation methods}\label{sec:numerical}

{\color{black}
We now test the numerical performance of the methods for approximating tensors norms, in complement to the theoretical results established earlier. All the experiments are conducted on a linux server (Ubuntu 20.04) with an Intel Xeon Platinum 8358 @ 2.60GHz and 512GB of RAM. The computations are implemented in Python 3. The semidefinite optimization solver\footnote{https://docs.mosek.com/latest/pythonfusion/tutorial-sdo-shared.html} in COPT Fusion API for Python 9.3.13 is called whenever semidefinite programs are involved.

We first test Algorithm~\ref{alg:snorm} to approximate the tensor spectral norm using examples in Nie and Wang~\cite[Examples 3.12, 3.13 and 3.14]{NW14}. The semidefinite relaxation method in~\cite{NW14} works well in practice and usually finds optimal values. This also enable us to check the true approximation bounds in practice rather than the conservative theoretical bounds. The results for the first two examples are shown in Table~\ref{table:00}. For Example 3.13, the method in~\cite{NW14} calls the fmincon function in MATLAB for a local improvement. This is the benchmark optimal value used to compute the approximation bounds. We also apply the classic alternating least square (ALS) method~\cite{KB09} as a local improvement starting from the approximate solutions obtained by Algorithm~\ref{alg:snorm}. Whenever a local improvement method is applied, the corresponding indicator is appended with a `+' sign.

\begin{table}[h]
\centering
\small
\setlength{\tabcolsep}{5pt}
        \centering
        \begin{tabular}{|l|l|cccccc|}
            \hline
Example &Method & CPU & CPU+ & Value  & Value+ & Bound & Bound+\\
            \hline
Ex 3.12&\cite{NW14} & 0.703 &  & 2.8167  &  & 1.0000& \\
&Alg~\ref{alg:snorm} & 0.000 & 0.000 & 2.2076 & 2.8167 & 0.7837 & 1.0000\\
            \hline
Ex 3.13&\cite{NW14} & 0.545 & 0.612 & 0.9862 & 1.0000 & 0.9862 & 1.0000\\
&Alg~\ref{alg:snorm} & 0.000 & 0.250 & 0.8397  & 1.0000 & 0.8397 & 1.0000\\
     \hline
\end{tabular}
\caption{Numerical results for Examples 3.12 and 3.13 in~\cite{NW14}}\label{table:00}
\end{table}

The results for Example 3.14 in~\cite{NW14} are shown in Table~\ref{table:01}. In this example, the method in~\cite{NW14} obtained global optimality directly without applying the local improvement. We also listed the theoretical approximation bound $\sqrt{\frac{\ln n}{n}}$ (without showing the constant disguised under the $\Omega$) of our algorithm for comparison.

\begin{table}[h]
\small
\setlength{\tabcolsep}{5pt}
\centering
\begin{tabular}{|l|l|ccccccc|}
\hline
$n$ &Method & CPU & CPU+ & Value  & Value+ & Bound & Bound+ & $\sqrt{\ln n /n}$\\
 \hline
5 & \cite{NW14} & 0.997 &  & 6.0996 &  & 1.0000 &  & \\
 & Alg~\ref{alg:snorm} & 0.020 & 0.050 & 4.3058 & 6.0996 & 0.7059 & 1.0000 & 0.5674\\\hline
10 & \cite{NW14} & 1.411 &  & 14.7902 &  & 1.0000 &  & \\
 & Alg~\ref{alg:snorm} & 0.320 & 1.920 & 8.4779 & 14.7902 & 0.5732 & 1.0000 & 0.4799\\\hline
15 & \cite{NW14} & 3.696 &  & 25.4829 &  & 1.0000 &  & \\
 & Alg~\ref{alg:snorm} & 1.670 & 3.680 & 11.4022 & 25.4829 & 0.4474 & 1.0000 & 0.4249\\\hline
20 & \cite{NW14} & 8.763 &  & 33.7020 &  & 1.0000 & & \\
 & Alg~\ref{alg:snorm} & 4.870 & 20.120 & 13.3617 & 33.7020 & 0.3964 & 1.0000 & 0.3870\\\hline
25 & \cite{NW14} & 37.535 &  & 46.7997 &  & 1.0000 &  & \\
 & Alg~\ref{alg:snorm} & 50.310 & 110.000 & 19.5674 & 46.7997 & 0.4181 & 1.0000 & 0.3588\\\hline
30 & \cite{NW14} & 52.994 &  & 64.9106 &  & 1.0000 &  & \\
 & Alg~\ref{alg:snorm} & 101.380 & 152.160 & 24.5234 & 64.9106 & 0.3778 & 1.0000 & 0.3367\\\hline
35 & \cite{NW14} & 111.547 &  & 80.7697 &  & 1.0000 &  & \\
 & Alg~\ref{alg:snorm} & 197.510 & 350.360 & 28.6220 & 80.7697 & 0.3543 & 1.0000 & 0.3187\\\hline
40 & \cite{NW14} & 241.565 &  & 95.0878 &  & 1.0000 &   & \\
 & Alg~\ref{alg:snorm} & 362.230 & 548.350 & 33.7020 & 95.0878 & 0.3307 & 1.0000  & 0.3037\\
                \hline
        \end{tabular}
\caption{Numerical results for Example 3.14 in~\cite{NW14}}\label{table:01}
\end{table}

Observed from the numerical results of these three examples, Algorithm~\ref{alg:snorm} obviously fails to obtain optimality in contrast to a practical method, but with the help of the ALS method the global optimality is obtained for all the test instances. The approximation bounds calculated by these numerical instances are better than the theoretical approximation bounds shown in Section~\ref{sec:snorm}. In terms of the computational time by comparing with the method in~\cite{NW14}, Algorithm~\ref{alg:snorm} runs quicker for low dimensions but the time increases quickly when the dimension of the problem increases.

To systematically verify and compare with the theoretical approximation bounds obtained by our algorithms, we now test randomly generated tensors whose spectral and nuclear norms can be easily obtained. In particular, let
\begin{equation} \label{eq:orthdecomp}
\TT=\sum_{i=1}^r \lambda_i\, \bx_i\otimes \by_i\otimes \bz_i \mbox{ with } \lambda_i>0 \mbox{ and } \|\bx_i\|=\|\by_i\|=\|\bz_i\|=1 \mbox{ for }i=1,2,\dots,r,
\end{equation}
where $(\bx_i^{\T}\bx_j)(\by_i^{\T}\by_j)=\bz_i^{\T}\bz_j=0$ if $i\neq j$. This is a special type of orthogonally decomposable tensors. With the special structure of $\TT$ in~\eqref{eq:orthdecomp}, it is not difficulty to see that $\|\TT\|_\sigma=\max_{1\le i\le r}\lambda_i$ and $\|\TT\|_*=\sum_{i=1}^r\lambda_i$. The components of $\TT$ in~\eqref{eq:orthdecomp}, $\lambda_i$'s, $\bx_i$'s, $\by_i$'s and $\bz_i$'s, are generated from i.i.d.~standard normal distributions and made positive (by taking the absolute value) or orthogonal if necessary.

We apply Algorithm~\ref{alg:snorm} to approximate the spectral norm and Algorithm~\ref{alg:3nnorm} to approximate the nuclear norm for $n\times 10\times 10$ tensors and $10\times n\times n$ tensors, both with varying $n$. Instead of the deterministic hitting set $\BH_5$ used in the original algorithms, we replace it with a randomized hitting set $\BH_1$ that is numerically more stable and efficient. The results are shown in Tables~\ref{table:4} and~\ref{table:6} for the spectral norm and in Tables~\ref{table:10} and~\ref{table:12} for the nuclear norm. For each type of tensors with a fixed size, say $5\times 10\times 10$, we randomly generate 200 instances and find an approximate solution of the spectral norm by Algorithm~\ref{alg:snorm}, whose approximation bound is then computed since the optimal value is known. We then use the approximate solution as a starting point to apply the ALS method as a local improvement. As before, the corresponding indicator is appended with a `+' sign when a local improvement is involved. 
The same setting is implemented for the tensor nuclear norm by Algorithm~\ref{alg:3nnorm} except that (1) there is no local improvement method to improve our approximation solution and (2) we do not multiply $\tau$ to the output solution $\Y$ as $\tau$ involves an $\Omega$ but directly use $\langle \TT,\Y \rangle$ to obtain the bound (see the proof of Theorem~\ref{thm:3nnorm}) and so the bound is larger than one. In this scenario, the closer to one the better the bound.

\begin{table}[h]
\small
        \setlength{\tabcolsep}{5pt}
        \centering
        \begin{tabular}{|l|cccccc|}
            \hline
            $n$  & 5 & 10 & 20 & 30 & 40 & 50\\
            \hline
            $\sqrt{\ln n/n}$ & 0.5674 & 0.4799 & 0.3870 & 0.3367 & 0.3037 & 0.2797 \\
            \hline
            Min bound & 0.6317 & 0.6344 & 0.5751 & 0.5263 & 0.4663 & 0.4602  \\
            Min bound+ & 0.6921 & 0.6500 & 0.5847 & 0.5371 & 0.4768 & 0.8603  \\
            Max bound  & 0.9879 & 0.9611 & 0.8572 & 0.7653 & 0.7025 & 0.6579  \\
            Max bound+ & 1.0000 & 1.0000 & 1.0000 & 1.0000 & 1.0000 & 1.0000  \\
            Mean bound  & 0.8898 & 0.8219 & 0.6786 & 0.6459 & 0.5778 & 0.5278  \\
            Mean bound+ & 0.9895 & 0.9896 & 0.9859 & 0.9825 & 0.9758 & 0.9932 \\
            \% of optimality+  & 92.0\% & 92.0\% & 91.0\% & 87.5\% & 84.5\% & 89.0\%  \\
            Mean CPU+ & 0.02 & 0.23 & 0.78 & 6.84 & 12.47 & 18.39  \\
            \hline
        \end{tabular}
        \caption{Approximating spectral norm by Algorithm 3.3 (using $\BH^n_1$) for $n\times 10\times 10$ tensors}\label{table:4}
\end{table}


\begin{table}[h]
\small
        \setlength{\tabcolsep}{5pt}
        \centering
        \begin{tabular}{|l|cccccc|}
            \hline
            $n$  & 5 & 10 & 20 & 30 & 40 & 50\\
            \hline
            Min bound & 0.6574 & 0.5016 & 0.5094 & 0.6858 & 0.5133 & 0.5905  \\
            Min bound+ & 0.6746 & 0.5321 & 0.5109 & 0.7236 & 0.5261 & 0.6099  \\
            Max bound  & 0.9451 & 0.9453 & 0.9472 & 0.9375 & 0.9819 & 0.9620  \\
            Max bound+ & 1.0000 & 1.0000 & 1.0000 & 1.0000 & 1.0000 & 1.0000  \\
            Mean bound  & 0.8271 & 0.8292 & 0.8308 & 0.8280 & 0.8326 & 0.8295  \\
            Mean bound+ & 0.9899 & 0.9826 & 0.9893 & 0.9933 & 0.9814 & 0.9851 \\
            \% of optimality+ & 90.0\% & 85.5\% & 91.0\% & 93.5\% & 88.5\% & 90.0\% \\
            Mean CPU+ & 0.06 & 0.23 & 0.76 & 1.81 & 3.12 & 5.53  \\
            \hline
        \end{tabular}
        \caption{Approximating spectral norm by Algorithm 3.3 (using $\BH^{10}_1$) for $10\times n\times n$ tensors}\label{table:6}
\end{table}

\begin{table}[H]
\small
        \setlength{\tabcolsep}{5pt}
        \centering
        \begin{tabular}{|l|cccccc|}
            \hline
            $n$  & 5 & 10 & 20 & 30 & 40 & 50\\
            \hline
            $\sqrt{n/\ln n}$& 1.7626 & 2.0840 & 2.5838 & 2.9699 & 3.2929 & 3.5751 \\
            \hline
            Min bound & 1.1791 & 1.2521 & 1.7417 & 1.8815 & 2.1672 & 2.5568 \\
            Max bound & 1.4998 & 1.5263 & 2.0248 & 2.0187 & 2.2854 & 3.2108 \\
            Mean bound & 1.3078 & 1.4135 & 1.9055 & 1.9522 & 2.2221 & 2.9763 \\
            Mean CPU+  & 0.69 & 13.31 & 101.49 & 1957.03 & 5365.99 & 11609.46 \\
            \hline
        \end{tabular}
\caption{Approximating nuclear norm by Algorithm 3.9 (using $\BH^{n}_1$) for $n\times 10\times 10$ tensors}\label{table:10}
\end{table}


\begin{table}[h]
\small
        \setlength{\tabcolsep}{5pt}
        \centering
        \begin{tabular}{|l|cccccc|}
            \hline
            $n$  & 5 & 10 & 20 & 30 & 40 & 50 \\
            \hline
            Min bound & 1.2999 & 1.3110 & 1.3008 & 1.3225 & 1.3638 & 1.3511 \\
            Max bound & 1.5275 & 1.5303 & 1.5257 & 1.5405 & 1.5099 & 1.5726 \\
            Mean bound & 1.4120 & 1.4112 & 1.4148 & 1.4148 & 1.4200 & 1.4239 \\
            Mean CPU+  & 2.06 & 13.23 & 160.03 & 941.46 & 6618.63 & 9488.85 \\
            \hline
        \end{tabular}
\caption{Approximating nuclear norm by Algorithm 3.9 (using $\BH^{10}_1$) for $10\times n\times n$ tensors}\label{table:12}
\end{table}

From the above tables, we see that the exact approximation bounds obtained by numerical instances outperform the theoretical bound $\Omega\left({\sqrt\frac{\ln n}{n}}\right)$  for both the spectral and nuclear norms. For the latter, it obviously beats the previous known best one $\Omega\left(\frac{1}{\sqrt{n}}\right)$. For the spectral norm, running the ALS method starting with our approximate solutions can lead to global optimality for most random generated tensor instances.
}

\section{Concluding remarks}\label{sec:remark}

We constructed hitting sets or collections of spherical caps to cover the unit sphere with adjustable parameters for different levels of approximations and cardinalities, listed roughly in Table~\ref{table:list}. These readily available products can be used for various decision making problems on spheres or related problems. By applying the covering results we proposed easily implementable and deterministic algorithms to approximate the tensor spectral norm with the current known best approximation bound. The algorithms can be extended to provide approximate solutions for sphere constrained homogeneous polynomial optimization problems. Deterministic algorithms with an improved approximation bound for the tensor nuclear norm were proposed as well. This newly improved bound attains the best known one for the tensor spectral norm. 

For $1\le p\le \infty$, the tensor spectral $p$-norm~\cite{L05} generalizes the tensor spectral norm in which the unit sphere $\|\bx\|=1$ is replaced by the $L_p$-sphere $\|\bx\|_p=1$. The tensor nuclear $p$-norm can also be defined similarly~\cite{FL18}. Hou and So~\cite{HS14} studied related $L_p$-sphere constrained homogeneous polynomial optimization problems and proposed approximation bounds. It is natural to ask whether one can construct $L_p$-sphere coverings and apply them to approximate the tensor spectral and nuclear $p$-norms. The answer is likely true but still challenging. In fact, one can construct randomized hitting sets using similar ideas in Section~\ref{sec:h1} to show the $L_p$ version of Theorem~\ref{thm:h1} but deterministic constructions remain difficult.
Perhaps a more interesting problem is to explicitly construct hitting sets for the binary hypercube $\{1,-1\}^n$ with different levels of approximations and cardinalities. It will have wider applications, particularly in discrete optimization and graph theory. We leave these to future works.

\section*{Acknowledgments}

The research is partially supported by the National Natural Science Foundation of China (Grants 71771141, 72171141, 71825003, 72150001, 72192832 and 11831002) and Program for Innovative Research Team of Shanghai University of Finance and Economics. The authors would like to thank the anonymous referees for their insightful comments that helped to improve this paper from its original version.

\end{document}